\documentclass[10pt,reqno]{amsart}
\usepackage{amsmath,amsfonts,amsthm,amscd,amssymb,graphicx}
\usepackage[utf8]{inputenc}
\usepackage{amsbsy}
\usepackage{graphicx}
\usepackage{hyperref}
\usepackage[margin=1in]{geometry}
\usepackage{subcaption}
\usepackage{enumitem}
\usepackage{float}
\usepackage{mathtools}
\usepackage{dsfont}

\usepackage{color}
\definecolor{myblue}{rgb}{.8, .8, 1}

\usepackage{amsmath}
\usepackage{empheq}

\newlength\mytemplen
\newsavebox\mytempbox

\makeatletter
\newcommand\mybluebox{%
    \@ifnextchar[
       {\@mybluebox}%
       {\@mybluebox[0pt]}}

\def\@mybluebox[#1]{%
    \@ifnextchar[
       {\@@mybluebox[#1]}%
       {\@@mybluebox[#1][0pt]}}

\def\@@mybluebox[#1][#2]#3{
    \sbox\mytempbox{#3}%
    \mytemplen\ht\mytempbox
    \advance\mytemplen #1\relax
    \ht\mytempbox\mytemplen
    \mytemplen\dp\mytempbox
    \advance\mytemplen #2\relax
    \dp\mytempbox\mytemplen
    \colorbox{myblue}{\hspace{1em}\usebox{\mytempbox}\hspace{1em}}}

\makeatother

\usepackage{tikz}
\usetikzlibrary{decorations.shapes}
\tikzset{paint/.style={ draw=#1!50!black, fill=#1!50 },
    decorate with/.style=
    {decorate,decoration={shape backgrounds,shape=#1,shape size=2mm}}}
\usepackage{todonotes}

\definecolor{skyblue}{rgb}{0.85,0.85,1}

\usepackage{accents}

\newtheorem{theorem}{Theorem}
\newtheorem{prop}[theorem]{Proposition}
\newtheorem{cor}[theorem]{Corollary}
\newtheorem{define}[theorem]{Definition}
\newtheorem{lemma}[theorem]{Lemma}
\newtheorem{rem}[theorem]{Remark}
\newtheorem{assumption}[theorem]{Assumption}
\newtheorem{conj}{Conjecture}

\def\Re{\mathop\mathrm{Re}\nolimits}    

\newcommand{\p}{\partial}				

\DeclareMathOperator{\sgn}{sgn}

\newcommand{\eqdef}{\vcentcolon =}

\numberwithin{equation}{section}
\numberwithin{theorem}{section}

\begin{document}
\title{Towards Spectral Convergence of Locally Linear Embedding on Manifolds with Boundary}
\author{Andrew Lyons}\email{ahlyons@email.unc.edu}\address{Dept. of Mathematics, UNC-CH, 418 Phillips Hall, Chapel Hill, NC 27599-3250, USA}

\begin{abstract}
We study the eigenvalues and eigenfunctions of a differential operator that governs the asymptotic behavior of the unsupervised learning algorithm known as Locally Linear Embedding when a large data set is sampled from an interval or disc. In particular, the differential operator is of second order, mixed-type, and degenerates near the boundary. We show that a natural regularity condition on the eigenfunctions imposes a consistent boundary condition and use the Frobenius method to estimate pointwise behavior. We then determine the limiting sequence of eigenvalues analytically and compare them to numerical predictions. Finally, we propose a variational framework for determining eigenvalues on other compact manifolds.
\end{abstract}
\maketitle

\section{Introduction} \label{sec:Intro}

\quad We study the eigenvalues and eigenfunctions of a mixed-type differential operator that informs the behavior of the dimensional limit of the \textit{Locally Linear Embedding} (LLE) matrix, as defined in \cite[Sec.$2$]{WW}. Locally Linear Embedding is a nonlinear algorithm used in data science to reduce the dimension of a data set and has been widely used in a variety of scientific fields. Introduced by Roweis and Saul in $2000$ \cite{SR}, the objective is to determine a map from a subset of $\mathbb{R}^{d_1}$ to $\mathbb{R}^{d_2}$ where $d_1>d_2$ and the local geometry is preserved, as made precise in Section \ref{sec:LLEmatrix}. The underlying idea is that by parametrizing a data set locally, one can recover another data set embedded in a lower dimensional space but with a similar local structure.

\quad The Locally Linear Embedding algorithm has a wide range of applications, from medical classification \cite{Med, SDD} to audiovisual processing \cite{TH, CFJ17}. Related algorithms, such as diffusion maps, have been used to numerically estimate solutions to differential equations on compact manifolds with boundary \cite{CL06,V20}, and LLE has the potential to do the same. This work furthers an effort to understand the differential operator approximated by the LLE matrix $W$ when a data set is sampled from a manifold with boundary and the number of data points is large \cite{WW18, WW}. The differential operator of interest is piecewise defined over a smooth, compact Riemannian manifold with boundary and degenerates on a submanifold of codimension $1$. The domain of the operator is not known a priori; however, we show how a regularity condition on the eigenfunctions enforces a boundary condition on the degenerate submanifold and present asymptotic properties of the eigenfunctions and eigenvalues. 

\quad Let $(M,g)$ be a $d$-dimensional, smooth, compact Riemannian manifold with smooth boundary, isometrically embedded in $\mathbb{R}^{d_1}$. To avoid error from sample variance, we take our high-dimensional data set $\{x_j\}_{j=1}^n$ along a uniform grid in $M$. Let $\epsilon$ be the algorithmic parameter that determines the neighborhood size about each $x_j$, as introduced in Section \ref{sec:LLEmatrix}. Smaller values of $\epsilon$ generally improve the local parametrization of each $x_j$ in LLE, but only if the number of data points $n$ is sufficiently large. To ensure that this happens, we hereafter take $\epsilon=\epsilon(n)$ to be supercritical, meaning 
\begin{equation}\label{eq:supercrit}
    \lim_{n\to\infty}\epsilon(n)=0  \quad \textrm{and} \quad 
    \lim_{n\to\infty} n\epsilon^{d_1}(n)=\infty.
\end{equation}
\pagebreak 

\quad In \cite{WW}, Wu and Wu found that, upon the regularization discussed in Section \ref{sec:LLEmatrix}, the LLE matrix $W\in\mathbb{R}^{n\times n}$ behaves like a second order, mixed-type differential operator as $n\to\infty$. More precisely, they proved that there exists a differential operator $D_\epsilon$ such that for $k\in\{1,\dots,n\}$,
\begin{equation}\label{eq:Pconv}
    \sum_{j=1}^n (I-W)_{kj} f(x_j)=\epsilon^2 D_\epsilon f(x_k)+O(\epsilon^3)
\end{equation}
with probability greater than $1-n^{-2}$. Note that we have deviated slightly from the presentation in \cite[Thm.$16$]{WW} via a global sign change and omission of the variance error.  The constant in the error term of (\ref{eq:Pconv}) depends on the $C^2$ norm of $f$ and properties of the manifold $M$. For this reason, we are interested in understanding eigenpairs $(\lambda, u)\in \mathbb{C}\times C^2(M)$ satisfying
\begin{equation}\label{eq:PDE}
    D_\epsilon u=\lambda u \quad \textrm{on} \quad M
\end{equation}
for small $\epsilon$. However, it is not immediately clear whether this objective is well-defined. One curious aspect of (\ref{eq:Pconv}) is the lack of apparent boundary conditions, which play no role in the finite-dimensional case. Motivated by \cite{EM13, LZ11}, we demonstrate how boundary conditions are naturally encoded in the function space $C^2(M)$ and that the eigenvalues can be defined via a min-max principle. 

\quad An additional challenge is posed by the fact that $D_\epsilon$ is mixed-type, changing from elliptic to hyperbolic in the boundary layer
\begin{equation}\label{eq:bdrylayer}
    M_\epsilon\eqdef \{x\in M: \textrm{dist}_g(x,\p M)<\epsilon\}
\end{equation}
as depicted in Figure \ref{fig:M}. This is remedied by studying the solution locally and patching local solutions to form a global one. For a given manifold $M$, we set
\begin{gather}
\begin{aligned}\label{eq:OmegaPM}
    M_+&=\{x\in M: D_\epsilon \textrm{ is elliptic at }x\} \\ M_-&=\{x\in M: D_\epsilon \textrm{ is hyperbolic at }x\}
\end{aligned}
\end{gather}
with $\Gamma$ denoting the boundary between the two regions, where $D_\epsilon$ degenerates. In particular, $\Gamma=\p M_+$ is diffeomorphic to the boundary $\p M$, and the hyperbolic region $M_-$ is a subset of $M_\epsilon$.

\begin{figure}[H]
    \centering
    \includegraphics[scale=0.25]{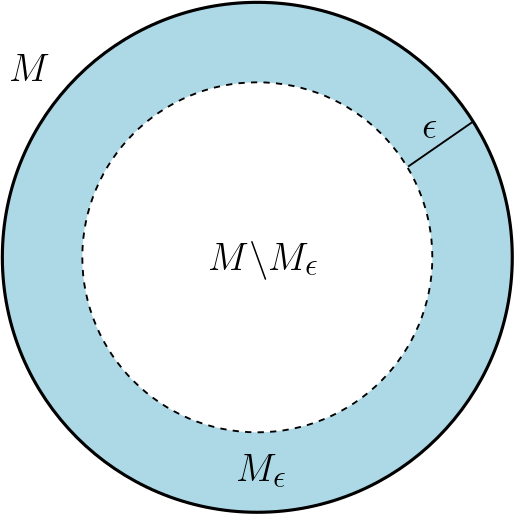}
    \hspace{28pt}
    \includegraphics[scale=0.25]{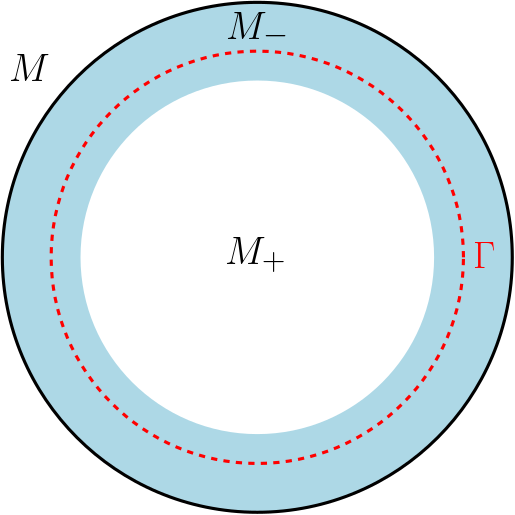}
    \caption{When $M$ is the disc, the left figure depicts the boundary layer $M_\epsilon$, and the right figure depicts the elliptic and hyperbolic regions $M_+,M_-$ separated by $\Gamma$.}
    \label{fig:M}
\end{figure}
\vspace{-10pt}

\quad A construction for $D_\epsilon$ is presented in \cite[Def. $12$]{WW}. However, because (\ref{eq:Pconv}) is pointwise, this construction relies on oriented normal coordinates at each point. For this reason, we analyze the differential operator on separable domains with a smooth boundary, where $D_\epsilon$ can be defined globally. On the interval $[0,1]$, (\ref{eq:PDE}) takes the Sturm-Liouville form
\begin{equation}\label{eq:1}
    -\left(p_1(x) u'(x)\right)'=\lambda w_1(x) u(x) \quad \textrm{for} \quad x\in(0,1),
\end{equation}
where exact formulas for $p_1,w_1$ are presented in Section \ref{sec:1}. The principal coefficient $p_1$ is continuous and sign-changing and hence defines the elliptic and hyperbolic regions in (\ref{eq:OmegaPM}). Meanwhile, the weight $w_1$ is continuous and positive over the interval. On the disc $\overline{B_1(0)}$ with polar coordinates, (\ref{eq:PDE}) takes the form
\begin{equation}\label{eq:2}
    -\frac{\p}{\p r}\left(p_2(1-r)\frac{\p u}{\p r}(r,\theta)\right)+q_2(1-r)\frac{\p^2u}{\p\theta^2}(r,\theta)=\lambda w_2(1-r)u(r,\theta),
\end{equation}
where exact formulas for $p_2,q_2,w_2$ are presented in Section \ref{sec:2}. As on the interval, the principal coefficient $p_2$ is continuous and sign-changing and defines (\ref{eq:OmegaPM}), while the weight $w_2$ is continuous and positive over the disc. In particular, $D_\epsilon$ degenerates where the principal coefficient vanishes. 

\quad For a general manifold, $D_\epsilon$ simplifies to the Laplacian outside the boundary layer, reducing (\ref{eq:PDE}) to
\begin{equation}
    -\frac{1}{2(d+2)}\Delta u=\lambda u \quad \textrm{on} \quad M\backslash M_\epsilon \nonumber
\end{equation}
where $d$ is the dimension of $M$. This is consistent with the LLE convergence result of \cite{WW18}, when $M$ is a compact manifold without boundary. Because $M_\epsilon$ vanishes as $\epsilon\to 0$, eigenfunctions of $D_\epsilon$ behave asymptotically like eigenfunctions of the Laplacian. Using the explicit construction of $D_\epsilon$ on the interval and disc, we determine a consistent boundary condition for the eigenfunctions of (\ref{eq:PDE}).

\quad Equations (\ref{eq:1}) and (\ref{eq:2}) introduce the relevant notion of a quasi-derivative, as presented in \cite{ZSZ14}. Absent any regularity requirement, there may be solutions to (\ref{eq:2}) for which the quasi-derivative $p_2(1-r) \frac{\p u}{\p r}(r,\theta)$ is differentiable but the derivative $\frac{\p u}{\p r}(r,\theta)$ is not. By construction, the partial derivative $\frac{\p}{\p r}$ serves as the normal derivative on $\Gamma$, so we interpret quasi-Neumann boundary conditions on $M_+$ as being satisfied when  $$p_2(1-r)\frac{\p u}{\p r} (r,\theta)=0 \quad \textrm{on} \quad \Gamma.$$ 

Using (\ref{eq:1}) and (\ref{eq:2}), we present eigenvalue and eigenfunction properties that arise in (\ref{eq:PDE}) under the condition $u\in C^2(M)$. We then present a natural extension of this framework to other compact manifolds with boundary.

\subsection{Eigenfunction properties on the interval and disc.}\label{sec:REsults} In this section, we present results regarding the limiting eigenvalue sequence of $D_\epsilon$ and discuss the natural boundary condition that corresponding eigenfunctions satisfy. We then present a variational interpretation of the eigenvalues that is consistent on both the interval and disc. The first result holds
when $D_\epsilon$ is defined on the unit interval and is proved in Section \ref{sec:pT1}. 

\begin{theorem}\label{thm:1d}
    Let $M$ be the unit interval $[0,1]$ and $M_+=\{x: p_1(x)>0\}$ with $p_1,w_1$ given in (\ref{eq:1}). There exists a constant $\epsilon_0>0$ such that for all $0<\epsilon\leq \epsilon_0$, the set of $C^2(M)$ solutions to (\ref{eq:PDE}) coincides with eigenfunctions of $D_\epsilon$ satisfying that for every $y\in\p M_+$,
    \begin{equation}\label{eq:quasiInt}
        \lim_{x\to y} p_1(x)u'(x)=0. 
    \end{equation}
    If (\ref{eq:quasiInt}) holds,
    \begin{enumerate}
        \item $D_\epsilon$ is self-adjoint on $L^2(M_+,w_1)$.
        \item The spectrum is discrete with $\sigma(D_\epsilon)\subset [0,\infty)$.
        \item Constant functions form the kernel of $D_\epsilon$. 
    \end{enumerate}
\end{theorem}

\quad This result states that the regularity requirement forces a Neumann-type boundary condition on a subset of $[0,1]$ that converges to the whole interval as $\epsilon$ vanishes. Further, the eigenvalue properties match those expected of the matrix $I-W$ when the sampled data set in LLE is sufficiently large, as discussed in Section \ref{sec:LLEmatrix}. For a precise estimate on the eigenvalues of $D_\epsilon$ on the interval, see Corollary \ref{prop:Est1}, and for a numerical demonstration of spectral convergence on the interval, see Section \ref{sec:3}. The second result holds when $D_\epsilon$ is defined on the unit disc and is proved in Section \ref{sec:pT2}. 

\begin{theorem}\label{thm:2d}
    Let $M$ be the unit disc $\overline{B_1(0)}$ and $M_+=\{(r,\theta): p_2(1-r)>0\}$ with $p_2,w_2$ given in (\ref{eq:2}). There exists a constant $\epsilon_0>0$ such that for all $0<\epsilon\leq \epsilon_0$, the set of $C^2(M)$ solutions to (\ref{eq:PDE}) coincides with eigenfunctions of $D_\epsilon$ satisfying that for every $y\in \p M_+$,
    \begin{equation}\label{eq:quasiDisc}
        \lim_{(r,\theta)\to y} p_2(1-r)\frac{\p u}{\p r}(r,\theta)=0.
    \end{equation}
    If (\ref{eq:quasiDisc}) holds,
    \begin{enumerate}
        \item $D_\epsilon$ is self-adjoint on $L^2(M_+,w_2)$.
        \item The spectrum is discrete with $\sigma(D_\epsilon)\subset [0,\infty)$.
        \item Constant functions form the kernel of $D_\epsilon$. 
    \end{enumerate}
\end{theorem}

\quad Similar to Theorem \ref{thm:1d}, this result states that the same regularity requirement forces a Neumann-type boundary condition on a subset of $\overline{B_1(0)}$ that converges to the whole disc as $\epsilon$ vanishes. The eigenvalues mirror the limiting properties of $I-W$ for large data sets. For a numerical demonstration of spectral convergence on the disc, see Section \ref{sec:3}. The following statement presents the variational framework for eigenvalues of $D_\epsilon$ on the interval (\ref{eq:1}) and is proved in Section \ref{sec:1}.

\begin{theorem}\label{thm:Energy}
    Let $M=[0,1]$ and set $$V=\left\{v\in L^2\left(M_+, w_1\right): \int_{M_+} p_1(x)\left|v'(x)\right|^2 dx<\infty\right\}.$$ If the eigenfunctions of $D_\epsilon$ on $M$ satisfy (\ref{eq:quasiInt}), then the eigenvalues $\{\lambda_j\}_j$ satisfy
    \begin{equation}
        \lambda_j=\max_{\substack{V_j\subseteq L^2(M_+,w_1), \\ |V_j|=j-1}}\left\{ \min_{v\in V\cap V_j^\perp\backslash\{0\}}\frac{\displaystyle \int_{M_+} p_1(x)\left|v'(x)\right|^2 dx}{\displaystyle \int_{M_+} w_1(x) \left|v(x)\right|^2 dx}\right\} \nonumber
    \end{equation}
    for all $j\geq 1$.
\end{theorem}

\quad This result states that the eigenvalues of $D_\epsilon$ can be determined via a min-max principle, in which an energy functional is varied over an appropriate function space. This method necessarily imposes quasi-Neumann boundary conditions on the minimizers and parallels the analogous finite-dimensional method for the matrix $I-W$. The next result holds when $D_\epsilon$ is defined on the unit disc and is proved in Section \ref{sec:2}.

\begin{theorem}\label{thm:Energy2}
    Let $M=\overline{B_1(0)}$ and set $$V=\left\{v\in L^2\left(M_+, w_2\right): \int_{M_+} p_2(1-r)\left|\frac{\p v}{\p r}(r,\theta)\right|^2 drd\theta+\int_{M_+}q_2(1-r)\frac{\p^2v}{\p\theta^2}(r,\theta) \overline{v(r,\theta)}drd\theta<\infty\right\}.$$  If eigenfunctions of $D_\epsilon$ on $M$ satisfy (\ref{eq:quasiDisc}), then the eigenvalues $\{\lambda_j\}_j$ satisfy
    \begin{equation}
        \lambda_j=\max_{\substack{V_j\subseteq L^2(M_+,w_2), \\ \left|V_j\right|=j-1}}\left\{ \min_{v\in V\cap V_j^\perp\backslash\{0\}}\frac{\displaystyle \int_{M_+} p_2(1-r)\left|\frac{\p v}{\p r}(r,\theta)\right|^2 drd\theta+\int_{M_+}q_2(1-r)\frac{\p^2v}{\p\theta^2}(r,\theta) \overline{v(r,\theta)}drd\theta}{\displaystyle \int_{M_+} w_2(1-r) \left|v(r,\theta)\right|^2 drd\theta}\right\} \nonumber
    \end{equation}
    for all $j\geq 1$.
\end{theorem}

\quad Analogous to Theorem \ref{thm:Energy}, this statement provides the variational framework for interpreting eigenvalues on the disc. Motivated by Theorems \ref{thm:Energy} and \ref{thm:Energy2}, we consider the following weak construction for $D_\epsilon$ when $M$ is a general compact, $d$-dimensional Riemannian manifold with smooth boundary.

\begin{define}\label{ass:form}
    We say $D_\epsilon$ is admissible on $M$ if there exist nonnegative functions $p,w$ on $M_+$ and a symmetric, nonnegative operator $A_\epsilon$ on $L^2(M_+\cap M_\epsilon,w)$ such that
        \begin{gather}\label{eq:Buv}
        \begin{aligned}
            \langle D_\epsilon u, v\rangle_{L^2(M_+,w)}=\frac{1}{2(d+2)}&\int_{M\backslash M_\epsilon} \nabla u(x)\overline{\nabla v(x)} dv_g \\ +&\int_{M_+\cap M_\epsilon} p(x_d) \frac{\p u}{\p x_d}(x) \overline{\frac{\p v}{\p x_d}}(x) dx\\ +& \int_{M_+\cap M_\epsilon} w(x_d) A_\epsilon u(x) \overline{v(x)} dx 
        \end{aligned}
        \end{gather}
        where $\{x_1,\dots,x_d\}$ are Fermi coordinates \cite[Ch. $2$]{AGT} such that $x_d=0$ along the boundary $\p M$ and $x_d>0$ in the interior of $M$. In addition,
        \begin{equation}\label{eq:ctssss}
            \frac{1}{2(d+2)}\int_{\p M_\epsilon}\frac{\p u}{\p \nu}(x)\overline{v(x)}d\sigma+p(\epsilon)\int_{\p M_\epsilon} \frac{\p u}{\p x_d}(x',\epsilon)\overline{v(x',\epsilon)}dx'=0
        \end{equation}
        where $x=(x',x_d)$ in $M_\epsilon$, $\nu$ is the outward-pointing normal on $\p M_\epsilon$ and $d\sigma$ is the measure on the boundary.

\end{define}

\quad $D_\epsilon$ is necessarily admissible on both the interval and disc. The function $p$ is meant to capture the degenerate behavior of $D_\epsilon$ in the boundary layer, remaining positive on $M_+$ and negative on $M_-$. As in Conjecture \ref{conj:1} below, we expect that the spectrum of $D_\epsilon$ is completely described by the behavior of the differential operator on the elliptic region. Equation (\ref{eq:ctssss}) is meant to stitch together the piecewise construction for $D_\epsilon$, which behaves like the Laplacian on $M\backslash M_\epsilon$ but relies on Fermi coordinates in $M_\epsilon$.

\quad Under Definition \ref{ass:form}, a variational interpretation of the eigenvalues enforces a quasi-Neumann boundary condition on the eigenfunctions of $D_\epsilon$ when defined over a general compact manifold. This is presented more
precisely in the following statement, which is proved in Section \ref{sec:6}.

\begin{theorem}\label{thm:blah}
    Let $M$ be a smooth, compact Riemannian manifold with smooth boundary and set 
    \begin{equation}
        V=\left\{v\in L^2(M_+,w): \langle D_\epsilon v,v\rangle_{L^2(M_+,w)}<\infty\right\}. \nonumber
    \end{equation}
    Let $D_\epsilon$ be admissible on $M$. If the eigenvalues $\{\lambda_j\}_j$ of $D_\epsilon$ satisfy
    \begin{equation}
        \lambda_j=\max_{\substack{V_j\subseteq L^2(M_+,w), \\ |V_j|=j-1}}\left\{ \min_{v\in V\cap V_j^\perp\backslash\{0\}}\frac{\langle D_\epsilon v,v\rangle_{L^2(M_+,w)}}{\langle v,v\rangle_{L^2(M_+,w)}}\right\}, \nonumber
    \end{equation}
    for all $j\geq 1$, then the corresponding eigenfunctions satisfy that for every $y\in\p M_+$,
    \begin{equation}\label{eq:quasiM}
        \lim_{x\to y} p(x_d)\frac{\p u}{\p x_d}(x)=0. 
    \end{equation}
\end{theorem}

\quad By construction, $-\frac{\p}{\p x_d}$ denotes the outward normal derivative along $\p M_+$, and $p$ is the principal coefficient that determines where $D_\epsilon$ is elliptic. Thus, $p$ vanishes along $\Gamma$, meaning that the quasi-Neumann boundary condition can be naturally interpreted as a statement regarding regularity. Namely, the conclusion of Theorem \ref{thm:blah} states that the normal derivative of $u$ does not blow up faster than $p$ vanishes.  

\quad In \cite{CL06}, Coifman and Lafon provided the leading order asymptotics for the Graph Laplacian $L\in\mathbb{R}^{n\times n}$ on manifolds with smooth boundary. In contrast to Wu and Wu's result for the LLE matrix \cite[Thm. $16$]{WW}, they found that $L$ behaved like a first order differential operator near the boundary, suggesting a Neumann boundary condition. This was later justified in \cite[Thm. $6.2$]{V20} when Vaughn, Berry, and Antil proved that the Graph Laplacian converges weakly to the Dirichlet form. Namely, $\sum_{k=1}^n\sum_{j=1}^n f(x_k) L_{kj}f(x_j)$ behaves like a multiple of $\int_{M}\left|\nabla f(x)\right|^2 dv_g$ for large $n$. We expect that, when sampled on a generic smooth, compact manifold with boundary, the matrix $I-W$ converges weakly to the form in (\ref{eq:Buv}). 

\begin{conj}\label{conj:1}
    Let $M$ be a Riemannian manifold with smooth boundary and suppose $\{x_j\}_{j=1}^n$ is sampled along a uniform grid in $M$. If $\epsilon=\epsilon(n)$ is supercritical, then $D_\epsilon$ is admissible on $M$ and the LLE matrix $W$ satisfies
    \begin{equation}
        \sum_{k=1}^n\sum_{j=1}^n  f(x_k)(I-W)_{kj}f(x_j)= \epsilon^2\frac{\langle D_\epsilon f,f\rangle_{L^2(M_+,w)}}{||f||^2_{L^2(M_+,w)}}+o(\epsilon^2). \nonumber
    \end{equation} 
\end{conj}

\quad In combination, Theorem \ref{thm:blah} and Conjecture \ref{conj:1} provide a justification for the quasi-Neumann boundary condition (\ref{eq:quasiM}) arising on general manifolds.

\subsection{Outline.} The structure of the paper is as follows. In Section \ref{sec:LLEmatrix}, we detail the Locally Linear Embedding algorithm and the construction of the LLE matrix $W$. We highlight key properties that support the results of Theorem \ref{thm:1d} and \ref{thm:2d}. In Section \ref{sec:1}, we detail the operator $D_\epsilon$ on the interval and use its form to determine the regularity of local solutions in the boundary layer. Using the Frobenius method, we isolate a unique eigenfunction and prove Theorem \ref{thm:1d}. We then present a variational framework for determining the spectrum and provide eigenvalue estimates. In Section \ref{sec:2}, we present $D_\epsilon$ on the unit disc and use a similar argument to prove Theorem \ref{thm:2d}. In particular, the boundary condition enforced by taking eigenfunctions in $C^2(M)$ is consistent on both domains.

\quad In Section \ref{sec:3}, we corroborate Theorems \ref{thm:1d} and \ref{thm:2d} numerically, demonstrating eigenvalue convergence on both the interval and disc. More precisely, by constructing a piecewise approximation for the eigenfunctions of (\ref{eq:PDE}), we use the matching conditions presented in Propositions \ref{prop:DetEq} and \ref{prop:DetEq2} to estimate the eigenvalues. In Section \ref{sec:6}, we prove Theorem \ref{thm:blah} under Definition \ref{ass:form}. Finally, in Appendix \ref{sec:appendix}, we discuss how the $C^2(M)$ regularity requirement might naturally arise from the higher order terms in (\ref{eq:Pconv}) when $M$ is the interval. We present an explicit differential operator that converges asymptotically to $D_\epsilon$ but enforces high regularity for small $\epsilon$. Motivated by this example, we conclude with a conjecture regarding the higher order terms in (\ref{eq:Pconv}) when $M$ is a general compact manifold with boundary.

\subsection{Acknowledgments.} The author is grateful to Yaiza Canzani, Jeremy Marzuola, and Hau-Tieng Wu for helpful conversations regarding the featured problem. The author received support from NSF grants DMS-$2045494$ and DMS-$1900519$ as well as from NSF RTG DMS-$2135998$.

\section{Construction and Properties of the LLE Matrix.}\label{sec:LLEmatrix} 

\quad In this section, we briefly describe how the Locally Linear Embedding matrix is constructed and its purpose in the LLE algorithm.  For a more detailed overview, we refer the reader to \cite[Sec. $2$]{WW18}. The goal of the algorithm is to determine a map
\begin{equation}
    \{x_j\}_{j=1}^n\subset \mathbb{R}^{d_1}\longrightarrow \{y_j\}_{j=1}^n\subset \mathbb{R}^{d_2}
\end{equation}
with $d_1>d_2$. By interpreting each data set as the discretization of a smooth Riemannian manifold, the algorithm was theoretically justified on compact manifolds without boundary in $2017$ \cite{MSWW19, WW18} and on compact manifolds with boundary in $2023$ \cite{WW}. Under this Riemannian manifold model, LLE unfolds a manifold locally to recover a global manifold embedded in a lower dimension. The local unfolding relies on parametrizing each data point as a linear combination of neighboring points and is encoded in a matrix $W\in\mathbb{R}^{n\times n}$ known as the LLE matrix.

\quad Given $\{x_j\}_{j=1}^n$ sampled from a compact manifold embedded in $\mathbb{R}^{d_1}$, one can construct the LLE matrix directly. The matrix is designed to capture the local geometry of a high-dimensional data set; its entries are preserved during the lower-dimensional embedding. The algorithm begins with a choice of nearest-neighbor scheme. Typically, one specifies a small, positive parameter $\epsilon$ and builds a neighborhood $N_\epsilon(x_j)$ about each $x_j$ that consists of all points $x_i\neq x_j$ such that $$||x_i-x_j||_{\ell^2(\mathbb{R}^{d_1})}<\epsilon.$$ However, we demonstrate the construction under the $k$-nearest neighbors scheme, which is numerically equivalent \cite[Sec. $5$]{WW18}. For a given $k\in\mathbb{N}$, we build a neighborhood $N_k(x_j)$ about each $x_j$ that consists of the $k$ closest data points, measured with the Euclidean metric on $\mathbb{R}^{d_1}$. In an effort to write each $x_j$ as an affine combination of its neighbors, we then minimize
\begin{equation}\label{eq:Wconstruct}
    \sum_{j=1}^n \bigg|\bigg|x_j-\sum_{x_i\in N_k(x_j)} w_{ij}x_i\bigg|\bigg|^2_{\ell^2(\mathbb{R}^{d_1})} 
\end{equation}
over all possible weights $w_{ij}\in\mathbb{R}$ subject to the linear constraint 
\begin{equation}\label{eq:constraint}
    \sum_i w_{ij}=1.
\end{equation}
\quad Interpreting $w_{ij}=0$ for $x_i\notin N_k(x_j)$, the barycentric coordinates $w_{ij}$ form the elements of the sparse matrix $W$. In practice, this step is explicit; one can construct local data matrices  $G_j\in\mathbb{R}^{d_1\times k}$ with columns given by the tangent vectors $x_i-x_j$ for $x_i\in N_k(x_j)$. The minimization of (\ref{eq:Wconstruct}) is then equivalent to minimizing the quadratic form $w^TG_j^TG_jw$ under the same linear constraint in (\ref{eq:constraint}).

\quad In general, the matrix $G_j^TG_j$ need not be invertible, meaning the algorithm needs to be stabilized. This can be done by introducing the regularizer $c>0$ and solving $$\left(G_j^TG_j+cI_{k\times k}\right)z_j=\textbf{1}_k$$ where $\textbf{1}_k$ is a $k\times 1$ vector with all entries equal to one. The weights then satisfy $$w_{ij}= \left(z_j^T\textbf{1}_k\right)^{-1}z_{ij}$$ where $z_{ij}$ is the $i$th element of $z_j$. This regularizer plays a crucial role in the asymptotics of LLE, as discussed in \cite{WW18}. In \cite{WW}, Wu and Wu set
\begin{equation}\label{eq:regularizer}
    c=n\epsilon^{d_1+3}
\end{equation}
where $n$ is the number of data points and $\epsilon$ is the neighborhood size. To ensure (\ref{eq:Pconv}) holds, we retain this choice throughout.

\quad The matrix $W\in\mathbb{R}^{n\times n}$ has several key properties. It satisfies $W\textbf{1}_n=\textbf{1}_n$ for all $n$ and has a spectral radius no less than $1$. It also becomes real and symmetric in the dimensional limit, as $n$ increases \cite[Prop. $2$]{WW}. Once $W$ is constructed, the bottom eigenvalues of the symmetric matrix $(I-W)^T(I-W)$ then provide the embedding coordinates for the low-dimensional data set $\{y_j\}_{j=1}^n$. When $n$ is large, this final step of the algorithm is computationally expensive, but $W$ is well-approximated by $D_\epsilon$ per (\ref{eq:Pconv}).

\section{LLE on the Interval}\label{sec:1}

\quad In this section, we study eigenpairs $(\lambda, u)$ satisfying (\ref{eq:PDE}) on the unit interval $[0,1]$. As introduced in \cite[Sec. $5$]{WW}, the differential operator $D_\epsilon$ is a second order differential operator with piecewise coefficients. First, we derive the Sturm-Liouville form and use the Frobenius method to understand the pointwise behavior of the eigenfunctions. We show that taking $u\in C^2([0,1])$ is equivalent to imposing quasi-Neumann boundary conditions on the eigenfunctions of $D_\epsilon$. Then we use the eigenfunction regularity to determine an eigenvalue condition. Finally, we discuss how to interpret the limiting eigenvalues variationally and prove Theorem \ref{thm:Energy} on the interval.

\subsection{Presentation of the differential operator on the interval.} While an exact definition of $D_\epsilon$ on the unit interval $M=[0,1]$ is featured in \cite[Cor. $17$]{WW}, we reproduce its form here for convenience. The boundary layer can be written $M_\epsilon=[0,\epsilon)\cup(1-\epsilon,1]$ with Fermi coordinates $$\begin{cases}x, & x\in[0,\epsilon) \\ 1-x, & x\in(1-\epsilon,1]. \end{cases}$$ 

\quad Due to this symmetry, we can globally define $D_\epsilon$ on $[0,1]$ using \cite[Def. $12$]{WW} in the form
\begin{equation}\label{eq:DeDef}
    D_\epsilon f(x)=\phi_2(x)f''(x)+\phi_1(x)f'(x)
\end{equation}
for $f\in C^2([0,1])$. Here each coefficient $\phi_1,\phi_2$ is constant outside the boundary layer. Namely, $$D_\epsilon f(x)=-\frac{1}{6}\Delta f(x)$$ when $x\in M\backslash M_\epsilon$. In the boundary layer, both coefficients are real-analytic and uniformly bounded. The leading coefficient can be written
\begin{align}
    \phi_2(x)&=\begin{cases} \frac{1}{12}\left(1-4\left(\frac{x}{\epsilon}\right)+\left(\frac{x}{\epsilon}\right)^2\right), & \textrm{if } x\in[0,\epsilon]; \\ -\frac{1}{6}, & \textrm{if } x\in[\epsilon, 1-\epsilon]; \\  \frac{1}{12}\left(1-4\left(\frac{1-x}{\epsilon}\right)+\left(\frac{1-x}{\epsilon}\right)^2\right), & \textrm{if } x\in[1-\epsilon,1].\end{cases} \nonumber
\end{align}
Note that $\phi_{2}$ is continuous on $[0,1]$, symmetric across the midpoint $x=\frac{1}{2}$, and vanishes near the boundary. If we set $$x_0=\left(2-\sqrt{3}\right)\epsilon,$$ then $D_\epsilon$ degenerates to a first order operator on $\{x_0, 1-x_0\}$. These singularities demand a careful analysis and motivate our derivation of the Sturm-Liouville form in (\ref{eq:ODE}). The remaining coefficient in (\ref{eq:DeDef}) can be written
\begin{align}    
    \phi_1(x)&=\begin{cases}  -6\left(1-\frac{x}{\epsilon}\right)\left(1+\frac{x}{\epsilon}\right)^{-3}, & \textrm{if } x\in[0,\epsilon]; \\ 0, & \textrm{if } x\in[\epsilon,1-\epsilon]; \\ 6\left(1-\frac{1-x}{\epsilon}\right)\left(1+\frac{1-x}{\epsilon}\right)^{-3}, & \textrm{if } x\in[1-\epsilon,1]. \end{cases} \nonumber
\end{align}
Similarly, $\phi_1$ is continuous on $[0,1]$ and has an odd symmetry across the midpoint $x=\frac{1}{2}$. To rewrite (\ref{eq:PDE}) in divergence form, we construct an integration factor via the function
\begin{equation}
    g(x)\eqdef\sgn(x-x_0)\left|x-x_0\right|^{\left(4+2\sqrt{3}\right)\epsilon}\left|x-4\epsilon+x_0\right|^{\left(4-2\sqrt{3}\right)\epsilon}\left(x+\epsilon\right)^{-8\epsilon} e^{\frac{12\epsilon^3}{(\epsilon+x)^2}+\frac{12\epsilon^2}{\epsilon+x}} \nonumber
\end{equation}
defined for $x\in[0,\epsilon]$. If we set
\begin{equation}
    p(x)=\begin{cases} \frac{g(x)}{6g(\epsilon)}, & \textrm{if } x\in[0,\epsilon]; \\ \frac{1}{6}, & \textrm{if } x\in[\epsilon,1-\epsilon]; \\ \frac{g(1-x)}{6g(\epsilon)}, & \textrm{if } x\in[1-\epsilon,1] \end{cases} \quad \textrm{and} \quad w(x)=-\frac{p(x)}{\phi_{2}(x)}  \nonumber
\end{equation}
then (\ref{eq:PDE}) can be written as the differential equation
\begin{equation}\label{eq:ODE}
    -\left(p(x) u'(x)\right)'=\lambda w(x) u(x)
\end{equation}
almost everywhere on $[0,1]$. In particular, (\ref{eq:ODE}) is a regular Sturm-Liouville equation\footnote{If we apply a global sign change to $\phi_1(x)$, then (\ref{eq:ODE}) is a singular Sturm-Liouville equation, which enforces a Dirichlet boundary condition on the elliptic region.}, with coefficient properties detailed in the following statement.

\begin{prop}\label{prop:SLproperties}
    For small enough $\epsilon$, the principal coefficient $p$ satisfies the following properties:
    \begin{enumerate}[label=\arabic*.]
        \item $p(x)$ is continuous and $\frac{1}{p}(x)$ is in $L^1([0,1])$.
        \item $p(x)$ is positive for $x\in(x_0,1-x_0)$, negative for $x\in[0,x_0)\cup (1-x_0,1]$,  and zero when $x=x_0$ or $x=1-x_0$.
        \item $p(x)$ is asymptotic to $\sgn(x-x_0)\left|x-x_0\right|^{(4+2\sqrt{3})\epsilon}$ as $x\to x_0$.
    \end{enumerate}
    The weight function $w$ satisfies the following properties:
    \begin{enumerate}[label=\arabic*.]
        \item $w(x)$ is continuous except at $x_0, 1-x_0$ and in $L^1([0,1])$.
        \item $w(x)$ is positive for all $x\in[0,1]$.
        \item $w(x)\to \infty$ as $x\to x_0$ or $x\to 1-x_0$.
    \end{enumerate}
    For all $x\in(0,1)$, $p(x)$ converges pointwise to $\frac{1}{6}$ and $w(x)$ converges pointwise to $1$ as $\epsilon\to 0$.
\end{prop}

\quad The changing sign of $p$ means that $D_\epsilon$ is mixed-type. Under (\ref{eq:OmegaPM}), (\ref{eq:DeDef}) and Proposition \ref{prop:SLproperties}, we identify the elliptic region as the set $M_+=\{x\in[0,1]: p(x)>0\}$ and the collection of interior singularities $\Gamma=\{x_0,1-x_0\}$ as the interface between elliptic and hyperbolic regions. With the differential operator $D_\epsilon$ constructed, we next establish properties of the eigenfunctions.

\subsection{Proof of Theorem \ref{thm:1d}.} \label{sec:pT1} In this section, we use the Frobenius method to study eigenfunctions of the differential operator presented in the previous section and prove Theorem \ref{thm:1d}. The regularity condition $u\in C^2([0,1])$ identifies a unique solution to (\ref{eq:ODE}) in the boundary layer, and matching conditions with the interior provide an exact eigenvalue condition. More precisely, if we define
\begin{equation}
    u_L(x)=u(x)\big|_{[0,\epsilon]}, \quad \quad u_M(x)=u(x)\big|_{[\epsilon, 1-\epsilon]}, \quad \textrm{ and } \quad u_R(x)=u(x)\big|_{[1-\epsilon, 1]}, \nonumber
\end{equation}
then (\ref{eq:ODE}) can be solved on each subinterval, and a global solution can be determined by gluing local solutions together. By symmetry of (\ref{eq:DeDef}), $u_L(x)$ are $u_R(1-x)$ are linearly dependent, so we begin our analysis near the left boundary, where (\ref{eq:ODE}) admits two local, linearly-independent solutions. Because the coefficients in (\ref{eq:DeDef}) are analytic over their respective domains, these solutions can be written as a Frobenius series centered at the singularity. The following result describes possible solution behavior in the leftmost boundary layer.

\begin{prop}\label{lem:0e}
    Let $(\lambda, u)$ satisfy (\ref{eq:ODE}) on the unit interval. The local solution $u_L(x)=u(x) \big|_{[0,\epsilon]}$ is a linear combination of two functions that we denote as $u_{\rm{Dir}}$ and $u_{\rm{Neu}}$. The solution $u_{\rm{Dir}}$ satisfies
    \begin{equation}
        u_{\rm{Dir}}(x_0)=0 \quad \textrm{and} \quad \lim_{x\to x_0}p(x)u'_{\rm{Dir}}(x)=1 \nonumber
    \end{equation}
    while the solution $u_{\rm{Neu}}$ satisfies
    \begin{equation}
        u_{\rm{Neu}}(x_0)=1 \quad \textrm{and} \quad \lim_{x\to x_0}p(x)u'_{\rm{Neu}}(x)=0. \nonumber
    \end{equation}
    In particular, $u_{\rm{Dir}}$ is continuous but not $C^1\left(0,\epsilon\right)$ while $u_{\rm{Neu}}$ is $C^\infty\left(0,\epsilon\right)$.
\end{prop}

\quad The notation for each basis function stems from their behavior at the singularity $x_0$. In Proposition \ref{cor:est} below, we include estimates for both functions, although the  condition $u\in C^2([0,1])$ discards any contribution from $u_{\rm{Dir}}$. Thus, the left solution $u_L$ is precisely the Neumann solution.

\begin{proof}[Proof of Proposition \ref{lem:0e}.]  The variable coefficients in (\ref{eq:DeDef}) are analytic and have a series representation centered at $x_0$. We write
\begin{equation}
    \phi_{2}(x)=\sum_{j=1}^2 a_j(x-x_0)^j \quad \textrm{with}\quad a_1=-\frac{1}{\sqrt{12}\epsilon}\quad \textrm{and} \quad a_2=\frac{1}{12\epsilon^2} \nonumber
\end{equation}
and
\begin{equation}
    \phi_1(x)=\sum_{j=0}^\infty b_j(x-x_0)^j \quad \textrm{with} \quad b_j=6\frac{(-1)^{j+1}}{\epsilon^j}\left(\frac{j^2+\sqrt{3}j+\sqrt{3}-1}{(3-\sqrt{3})^{j+3}}\right) \nonumber
\end{equation}
for all $j\geq 0$. Any local solution on $(0,\epsilon)$ takes the form of a Frobenius series
\begin{equation}\label{eq:series}
    u(x)=\sum_{j= 0}^\infty c_j\left(x-x_0\right)^{j+\alpha}
\end{equation} 
for $\alpha\in\mathbb{R}$. The indicial polynomial 
\begin{equation}\label{eq:indPoly}
    P(\alpha)= a_1\alpha(\alpha-1)+b_0\alpha
\end{equation}
has two roots at $0$ and $1-(4+2\sqrt{3})\epsilon$. Because these roots differ by a non-integer for small enough $\epsilon$, there are two linearly independent solutions to (\ref{eq:ODE}) near the left boundary, each with series given by(\ref{eq:series}). Without loss of generality, we initially take $c_0=1$ for both solutions. For $j\geq1$, the remaining Frobenius coefficients can be determined through the recurrence relation
\begin{equation}\label{eq:RecRel}
    P(\alpha+j)c_j=\left(\lambda-a_2(\alpha+j-1)(\alpha+j-2)\right) c_{j-1}-\sum_{k=0}^{j-1}c_k (k+\alpha)b_{j-k}. 
\end{equation}
\quad When $\alpha=1-(4+2\sqrt{3})\epsilon$, we denote the series (\ref{eq:series}) as the Dirichlet solution so that $u_{\rm{Dir}}(x_0)=0$. Proposition \ref{prop:SLproperties} then implies that the quasi-derivative $p_1(x)u_{\rm{Dir}}'(x)$ has a finite, nonzero limit as $x$ approaches $x_0$. When $\alpha=0$, we denote the series (\ref{eq:series}) as the Neumann solution so that $p_1(x_0)u_{\rm{Neu}}'(x_0)=0$. By scaling both solutions appropriately, we have the desired result.
\end{proof}

\quad Proposition \ref{lem:0e} implies that the regularity condition $u\in C^2([0,1])$ is equivalent to imposing quasi-Neumann boundary conditions on the elliptic region $M_+$. This necessarily means that the spectrum is discrete \cite[Ch. $5$]{N68}, and the eigenvalues of $D_\epsilon$ can be written as
\begin{equation}\label{eq:RRQuotient}
    \lambda=\frac{\displaystyle\int_{x_0}^{1-x_0} p(x)\left|u'(x)\right|^2dx }{\displaystyle\int_{x_0}^{1-x_0} w(x)\left|u(x)\right|^2dx }
\end{equation}
for $u$ satisfying (\ref{eq:ODE}). By Proposition \ref{prop:SLproperties} alongside (\ref{eq:RRQuotient}), the eigenvalues are real and nonnegative and the eigenfunctions are orthogonal with respect to the $L^2\left(M_+, w\right)$ inner product. Because $w(x)$ converges pointwise to $1$ as $\epsilon\to 0$, this corresponds to the fact that the LLE matrix is symmetric and real in the dimensional limit. Further, when $\lambda=0$, the remaining Neumann solution $u_{\rm{Neu}}$ is a constant. By (\ref{eq:Pconv}), this verifies that constant vectors are eigenvectors of the LLE matrix and completes the proof of Theorem \ref{thm:1d}.

\quad To estimate eigenvalues via (\ref{eq:RRQuotient}) would require a complete understanding of the related eigenfunctions over the elliptic region. Because $D_\epsilon$ is defined piecewise in (\ref{eq:DeDef}), we can alternatively characterize the eigenvalues by the continuity conditions imposed at $\epsilon$ and $1-\epsilon$. The following estimates follow from the proof of Proposition \ref{lem:0e}. 

\begin{prop}\label{cor:est}
    Let $(\lambda, u)$ satisfy (\ref{eq:ODE}) on $[0,1]$ with $$\lim_{\epsilon\to 0}|\lambda(\epsilon)| \epsilon^2=0.$$ There exist constants $C$ and $\epsilon_0$ such that if $0<\epsilon\leq \epsilon_0$, then $u_{\rm{Dir}}$ satisfies
    \begin{equation}
        \left|u_{\rm{Dir}}(\epsilon)\right|\leq C\epsilon \quad \textrm{and} \quad \left|u_{\rm{Dir}}'(\epsilon)-1\right|\leq C\epsilon\left(1+\epsilon|\lambda|\right) \nonumber
    \end{equation}
    and $u_{\rm{Neu}}$ satisfies
    \begin{equation}
        \left|u_{\rm{Neu}}(\epsilon)-1\right|\leq C\epsilon |\lambda| \quad \textrm{and} \quad \left|u_{\rm{Neu}}'(\epsilon)-\left(3-2\sqrt{3}\right)\lambda \right|\leq C\epsilon |\lambda|\left(1+\epsilon|\lambda|\right). \nonumber
    \end{equation}
\end{prop}

\quad Because (\ref{eq:ODE}) has a spectral parameter, the coefficients in (\ref{eq:series}) depend explicitly on the eigenvalue; the condition on $\lambda$ avoids complications from this dependence. Additionally, the terms in the respective series for $u_{\rm{Dir}}$ and $u_{\rm{Neu}}$ lose any sense of ordering with respect to $\epsilon$ for $|\lambda|$ comparable to $\epsilon^{-2}$. Because LLE utilizes only the smallest eigenvalues in the spectrum, as indicated in Section \ref{sec:LLEmatrix}, this condition is justifiable.

\begin{proof}[Proof of Proposition \ref{cor:est}.] By taking $|\lambda| \epsilon^2\to 0$ as $\epsilon\to 0$, we ensure that the constants in Proposition \ref{cor:est} are independent of the eigenvalue. As in Proposition \ref{lem:0e}, there are two solutions to consider, each corresponding to a different root of the indicial polynomial (\ref{eq:indPoly}). Retaining the notation from the proof of Proposition \ref{lem:0e}, there exists a constant $C$ such that 
\begin{equation}\label{eq:bjbd}
    |b_j|\leq Cj^{-1}\epsilon^{-j}
\end{equation}
for all $j\geq 1$. We break this proof into two cases depending on the root $\alpha$.

\textbf{Case 1.} Consider $u_{\rm{Dir}}$ when $\alpha=1-(4+2\sqrt{3})\epsilon$. The first coefficient $c_1$ is bounded by a multiple of $1+\epsilon|\lambda|$, and because $|\lambda|\epsilon^2=o(1)$, (\ref{eq:RecRel}) and (\ref{eq:bjbd}) inductively imply that $$|c_j|\leq \epsilon^{1-j}|c_1|$$ for all $j\geq 1$. We expand
\begin{equation}
    u_{\rm{Dir}}(\epsilon)=\left(\epsilon-x_0\right)^{1-(4+2\sqrt{3})\epsilon}+\left(\epsilon-x_0\right)^{1-(4+2\sqrt{3})\epsilon}\sum_{j=1}^\infty c_j(\epsilon-x_0)^j, \nonumber
\end{equation}
noting that the first term above is bounded by a multiple of $\epsilon$ and the second term can be bounded
\begin{equation}\label{eq:bd1}
    \bigg|(\epsilon-x_0)^{1-(4+2\sqrt{3}\epsilon)}\sum_{j=1}^\infty c_j(\epsilon-x_0)^j\bigg|\leq C \epsilon^2|c_1|\sum_{j=1}^\infty \left|1-\frac{x_0}{\epsilon}\right|^j. 
\end{equation}
Because $\left|1-\frac{x_0}{\epsilon}\right|<1$, (\ref{eq:bd1}) features a convergent geometric series and is bounded by a multiple of $\epsilon^2(1+\epsilon|\lambda|)$. Again because $|\lambda|\epsilon^2=o(1)$, this provides the pointwise bound for $u_{\rm{Dir}}$. The derivative estimate holds similarly.

\textbf{Case 2.} Consider $u_{\rm{Neu}}$ when $\alpha=0$. The first coefficient $c_1$ is equal to $\left(3-2\sqrt{3}\right)\lambda $, but we need to consider coefficients of higher indices. The second coefficient $c_2$ is bounded by a multiple of $|\lambda|\left(1+\epsilon|\lambda|\right)$ and because $\lambda\epsilon^2=o(1)$, (\ref{eq:RecRel}) and (\ref{eq:bjbd}) inductively imply that $$|c_j|\leq \epsilon^{2-j}|c_2|$$ for all $j\geq 2$. We expand
\begin{equation}
    u_{\rm{Neu}}(\epsilon)=1+c_1(\epsilon-x_0)+\sum_{j=2}^\infty c_j(\epsilon-x_0)^j, \nonumber
\end{equation}
noting that the second term above is bounded by a multiple of $\epsilon |\lambda|$ and the third term can be bounded
\begin{equation}\label{eq:bd2}
   \bigg| \sum_{j=2}^\infty c_j(\epsilon-x_0)^j\bigg|\leq C\epsilon^2|c_2|\sum_{j=2}^\infty\left|1-\frac{x_0}{\epsilon}\right|^j.
\end{equation}
As in the previous case, the geometric series is convergent and (\ref{eq:bd2}) is bounded by a multiple of $\epsilon^2|\lambda|\left(1+\epsilon|\lambda|\right)$. Because $|\lambda|\epsilon^2=o(1)$, this provides the pointwise bound for $u_{\rm{Neu}}$. The derivative estimate holds similarly.
\end{proof}

\quad Because we require the eigenfunctions to be sufficiently regular, the spectral behavior of $D_\epsilon$ can be estimated via matching conditions and eigenfunction estimates. This results in the following statement, which identifies eigenvalues as roots of an explicit transcendental equation.

\begin{prop}\label{prop:DetEq}
    Let $(\lambda, u)\in [0,\infty)\times C^2([0,1])$ satisfy (\ref{eq:ODE}). If $u_L(x)=u(x)\big|_{[0,\epsilon]}$  and $\epsilon$ is sufficiently small, then the eigenvalue $\lambda$ satisfies
    \begin{equation}
        \sin\left(\sqrt{6\lambda}\left(1-2\epsilon\right)\right)\left(u_L'^2(\epsilon)-6\lambda u_L^2(\epsilon)\right)+2\sqrt{6\lambda}\cos\left(\sqrt{6\lambda}\left(1-2\epsilon\right)\right)\left(u_L(\epsilon)u_L'(\epsilon)\right)=0. \nonumber
    \end{equation}
\end{prop}

\quad This result effectively describes the process of gluing local solutions to form a global eigenfunction. The regularity condition not only isolates a unique solution in the boundary layer, as highlighted in Proposition \ref{lem:0e}, but it also enforces this matching condition.

\begin{proof}[Proof of Proposition \ref{prop:DetEq}.] Because $u\in C^2([0,1])$, the eigenvalue $\lambda$ must simultaneously verify
\begin{equation}
    \lim_{x\to \epsilon^-} u_L(x)=\lim_{x\to\epsilon^+}u_M(x) \quad \textrm{  and  } \quad \lim_{x\to \epsilon^-} u_L'(x)=\lim_{x\to\epsilon^+}u_M'(x) \nonumber
\end{equation}
along with 
\begin{equation}
    \lim_{x\to \epsilon^-} u_M(x)=\lim_{x\to\epsilon^+}u_R(x) \quad \textrm{  and  } \quad \lim_{x\to \epsilon^-} u_M'(x)=\lim_{x\to\epsilon^+}u_R'(x). \nonumber
\end{equation}
Note that by (\ref{eq:DeDef}), these two conditions guarantee continuity of the second derivative at $\epsilon$. The local solution $u_M$ is a Laplacian eigenfunction with frequency $\sqrt{6\lambda}$, i.e. a linear combination of sines and cosines. For $u$ to be nontrivial, the following matrix 
\begin{equation}
    \begin{pmatrix} u_L(\epsilon) & 0 & -1 & 0 \\ 0 & -\sin\left(\sqrt{6\lambda}\left(1-2\epsilon\right)\right) & -\cos\left(\sqrt{6\lambda}\left(1-2\epsilon\right)\right) & u_L(\epsilon) \\ u_L'(\epsilon) & -\sqrt{6\lambda} & 0 & 0 \\ 0 & -\sqrt{6\lambda}\cos\left(\sqrt{6\lambda}\left(1-2\epsilon\right)\right) & \sqrt{6\lambda}\sin\left(\sqrt{6\lambda}\left(1-2\epsilon\right)\right) & -u_L'(\epsilon)\end{pmatrix} \nonumber
\end{equation}
must have zero determinant. This condition equates eigenvalues to roots of the desired transcendental equation.
\end{proof}

\quad In combination, Proposition \ref{cor:est} and Proposition \ref{prop:DetEq} prove the following first order estimate for the eigenvalues of $D_\epsilon$ on the interval.

\begin{cor}\label{prop:Est1}
    Let $u\in C^2([0,1])$ satisfy (\ref{eq:PDE}) on $[0,1]$ with eigenvalue $\lambda=\lambda(\epsilon)$. If $\lim_{\epsilon\to 0}|\lambda|\epsilon^2=0,$ then there exist positive constants $C,\epsilon_0$ such that for all $0<\epsilon\leq\epsilon_0$,
    \begin{equation}
        \left|\sin\left(\sqrt{6\lambda}\right)\left(\left(21-12\sqrt{3}\right)\lambda-6\right)+2\sqrt{6\lambda}\left(3-2\sqrt{3}\right)\cos\left(\sqrt{6\lambda}\right)\right|\leq C\epsilon \left|\lambda\right|^{\frac{3}{2}}. \nonumber
    \end{equation}
\end{cor}

\quad By taking $\lambda(\epsilon)$ smaller than $\epsilon^{-\rho}$ for any $\rho<\frac{2}{3}$ and letting $\epsilon$ vanish, we close the bootstrapping argument and recover a complete description of the eigenvalues of the limiting operator. Note that the error increases further in the spectrum. The eigenvalues also share a variational interpretation under the requirement that each eigenfunction be $C^2([0,1])$. If we define the energy functional
\begin{equation}\label{eq:E1}
    E(f)= \frac{\displaystyle\int_{x_0}^{1-x_0} p(x)\left|f'(x)\right|^2dx}{\displaystyle\int_{x_0}^{1-x_0} w(x) \left|f(x)\right|^2 dx} 
\end{equation}
over all $f$ for which $E(f)$ is finite, then the minimizers of $E$ on orthogonal subspaces of $L^2\left(M_+, w\right)$ are eigenfunctions of (\ref{eq:ODE}) with Neumann boundary conditions \cite[Ch. $6$]{CH} in the sense that $$\lim_{x\to x_0}p(x)f'(x)=0.$$ In this manner, the eigenvalues satisfy an infinite-dimensional min-max principle analogous to the Rayleigh min-max principle for matrices. This proves Theorem \ref{thm:Energy} on the interval.

\begin{rem}
    In \cite[Sec. $7$]{WW}, Wu and Wu introduce the clipped LLE matrix, a useful tool for numerically estimating eigenfunctions of the Dirichlet Laplacian on manifolds in which the elliptic region (\ref{eq:OmegaPM}) is known a priori. By Proposition \ref{lem:0e}, this corresponds directly to setting $u_L(x)=u_{\rm{Dir}}(x)$ in the left boundary layer. In combination, Propositions \ref{cor:est} and \ref{prop:DetEq} provide estimates for eigenvalue convergence of the clipped LLE matrix.
\end{rem}

\section{LLE on the Disc}\label{sec:2}

\quad In this section, we study eigenpairs $(\lambda, u)$ satisfying (\ref{eq:PDE}) on the unit disc $M=\overline{B_1(0)}$. First, we present an explicit expression for $D_\epsilon$ on the disc and find its Sturm-Liouville form. Following the same argument as in Section \ref{sec:1}, we then isolate a unique eigenfunction in the boundary layer and use regularity to determine an exact eigenvalue condition. Finally, we prove Theorem \ref{thm:2d} and present the eigenvalues in a variational format.
  
\subsection{Presentation of the differential operator on the disc.} In this section, we use \cite[Def. $12$]{WW} to explicitly derive the differential operator $D_\epsilon$ on the unit disc. The boundary layer can be written $$M_\epsilon=\{(r,\theta): 1-\epsilon<r\leq1\}$$ with Fermi polar coordinates $(1-r,\theta)$. Due to the radial symmetry, there exist differential operators $\p_{11}^2,\p_{22}^2,\p_2$ such that $D_\epsilon$ can be globally defined on the disc with the form
\begin{equation}\label{eq:DeDef2}
    D_\epsilon f(r,\theta) =\phi_{11}(1-r)\p_{11}^2f(r,\theta)+\phi_{22}(1-r)\p_{22}^2f(r,\theta)+\phi_2(1-r)\p_2f(r,\theta)
\end{equation}
for $f\in C^2\big(\overline{B_1(0)}\big)$. The variable coefficients in (\ref{eq:DeDef2}) are defined in terms of sigma functions that capture the geometric asymmetries of a neighborhood of each point. Each coefficient is constant outside the boundary layer, reducing to
\begin{equation}
    D_\epsilon f(x)=-\frac{1}{8}\Delta f(x) \nonumber
\end{equation}
for $x\in M\backslash M_\epsilon$.

\quad We can write the leading order coefficients as
\begin{align}
    \phi_{11}(t)&=\frac{1}{2}\frac{\sigma_{2,2}(t)\sigma_2(t)-\sigma_3(t)\sigma_{1,2}(t)}{\sigma_{1,2}^2(t)-\sigma_{2,2}(t)\sigma_{0}(t)}  \nonumber \\
    \phi_{22}(t)&=\frac{1}{2}\frac{\sigma_{2,2}^2(t)-\sigma_{3,2}(t)\sigma_{1,2}(t)}{\sigma_{1,2}^2(t)-\sigma_{2,2}(t)\sigma_{0}(t)} \nonumber
\end{align}
and the remaining coefficient as
\begin{equation}
    \phi_2(t)=\frac{\sigma_{1,2}(t)}{\sigma_{1,2}^2(t)-\sigma_{2,2}(t)\sigma_{0}(t)}. \nonumber
\end{equation}
The sigma functions can be derived explicitly for a given dimension and have a common geometric interpretation, as discussed in \cite[Def. $7$]{WW}. In dimension $2$, they are defined as follows:
{\allowdisplaybreaks
\begin{align}
    \sigma_{0}(t)&=\begin{cases} \frac{\pi}{2}+\frac{t}{\epsilon}\sqrt{1-\left(\frac{t}{\epsilon}\right)^2}+\arcsin\left(\frac{t}{\epsilon}\right), & 0\leq t\leq \epsilon \\ \pi, & \textrm{otherwise} \end{cases} \nonumber
    \\
    \sigma_{1,2}(t)&=\begin{cases} -\frac{2}{3}\left(1-\left(\frac{t}{\epsilon}\right)^2\right)^{\frac{3}{2}}, & 0\leq t\leq \epsilon \\ 0, & \textrm{otherwise} \end{cases} \nonumber
    \\
    \label{eq:sigma}
    \begin{split}
    \sigma_2(t)&=\begin{cases} \frac{\pi}{8}+\frac{1}{12}\left(\frac{t}{\epsilon}\sqrt{1-\left(\frac{t}{\epsilon}\right)^2}\left(5-2\left(\frac{t}{\epsilon}\right)^2\right)+3\arcsin\left(\frac{t}{\epsilon}\right)\right), & 0\leq t\leq \epsilon \\ \frac{\pi}{4}, & \textrm{otherwise} \end{cases} 
    \\
    \sigma_{2,2}(t)&=\begin{cases} \frac{\pi}{8}+\frac{1}{4}\left(\frac{t}{\epsilon}\sqrt{1-\left(\frac{t}{\epsilon}\right)^2}\left(2\left(\frac{t}{\epsilon}\right)^2-1\right)+\arcsin\left(\frac{t}{\epsilon}\right)\right), & 0\leq t\leq \epsilon \\ \frac{\pi}{4}, & \textrm{otherwise} \end{cases} 
    \end{split}
    \\
    \sigma_3(t)&=\begin{cases} -\frac{2}{15}\left(1-\left(\frac{t}{\epsilon}\right)^2\right)^{\frac{5}{2}}, & 0\leq t\leq \epsilon \\ 0, & \textrm{otherwise} \end{cases} \nonumber
    \\
    \sigma_{3,2}(t)&=\begin{cases} -\frac{2}{15}\left(2+3\left(\frac{t}{\epsilon}\right)^2\right)\left(1-\left(\frac{t}{\epsilon}\right)^2\right)^{\frac{3}{2}}, & 0\leq t\leq \epsilon \\ 0, & \textrm{otherwise}. \end{cases} \nonumber
\end{align}
}

\quad Notice that each sigma function above is analytic on $(0,\epsilon)$ but cannot extend analytically beyond $t=\epsilon$. This contrasts the leading coefficient of $D_\epsilon$ on the interval, as presented in Section \ref{sec:1}. While these sigma functions are crucial for the numerical analysis in Section \ref{sec:3}, we only use a few key properties in this section. The following statement follows directly from (\ref{eq:sigma}) and \cite[Prop. $13$]{WW}.

\begin{prop} \label{prop:phi1phi2V}
For small enough $\epsilon$, the variable coefficients in (\ref{eq:DeDef2}) have the following properties. 
    \begin{enumerate}[label=\arabic*.]
        \item $\phi_{11}(t)$ is negative  for all $t\geq 0$ and $\phi_{11}(t)=-\frac{1}{8}$ for $t\geq \epsilon$.

        \item $\phi_{22}(t)=-\frac{1}{8}$ for $t\geq\epsilon$.
        \item $\phi_{22}(0)>0$ and there exists a unique $r_0\in(0,\epsilon)$ such that $\phi_{22}(r_0)=0$.
        \item $\phi_{22}(t)$ vanishes linearly at $r_0$, in the sense that $\phi_{22}'(r_0)\neq 0$.

        \item $\phi_2(t)$ is nonnegative for all $t\geq0$ and $\phi_2(t)=0$ for $t\geq\epsilon$.
    \end{enumerate}
    Additionally, each coefficient is continuous and uniformly bounded on $[0,1]$ and analytic except at $t=\epsilon$.
\end{prop}

\quad Although each function in Proposition \ref{prop:phi1phi2V} can be written explicitly using (\ref{eq:sigma}), the featured properties are sufficient for proving Theorem \ref{thm:2d}. Because $\phi_{22}(1-r)$ vanishes in the boundary layer, the elliptic and hyperbolic regions (\ref{eq:OmegaPM}) are separated by a circle of radius $1-r_0$, where $r_0$ is the unique root of $\phi_{22}$.

\quad To present $D_\epsilon$ in its entirety, we lastly need to present the differential operators $\p_{11}^2, \p_{22}^2, \p_2$ in (\ref{eq:DeDef2}). We define
\begin{equation}\label{eq:NormalCoord}
    \p_{11}^2=\frac{1}{r}\frac{\p}{\p r}+\frac{1}{r^2}\frac{\p^2}{\p\theta^2}, \quad \p_{22}^2=\frac{\p^2}{\p r^2}, \quad \textrm{and} \quad \p_2=\frac{\p}{\p r} 
\end{equation}
so that (\ref{eq:DeDef2}) verifies the pointwise construction featured in \cite[Def. $12$]{WW}. Namely, for a fixed $x$ in the boundary layer,
\begin{equation}
    \p_if(x)=\frac{d}{dt}\left(f\circ \gamma_i\right)(0) \nonumber
\end{equation}
where $\gamma_i(t)$ is a geodesic satisfying $\gamma_i(0)=x$.
In particular, $\p_{11}^2+\p_{22}^2$ is the polar Laplacian and $\p_2$ is oriented towards the boundary. Equations (\ref{eq:sigma}) and (\ref{eq:NormalCoord}) provide the polar representation of $D_\epsilon$. 

\quad The fact that $D_\epsilon$ degenerates follows from the properties of $\phi_{22}$ in Proposition \ref{prop:phi1phi2V} and is more clear by writing (\ref{eq:PDE}) in the separable Sturm-Liouville form
\begin{equation}\label{eq:SL2}
    -\frac{\p}{\p r}\left(p(1-r)\frac{\p u}{\p r}(r,\theta)\right)+q(1-r)\frac{\p^2u}{\p\theta^2}(r,\theta)=\lambda w(1-r)u(r,\theta).
\end{equation}
Here the principal coefficient $p$ is defined as
\begin{equation}\label{eq:p2}
    p(1-r)=\frac{1}{8}(1-\epsilon)\sgn\left(1-r_0-r\right)\exp\left(\int_{1-\epsilon}^{r} \frac{\phi_{11}(1-s)+s\phi_2(1-s)}{s\phi_{22}(1-s)} ds\right) 
\end{equation}
which is sign-changing and vanishes when $r=1-r_0$. The remaining coefficients are defined as
\begin{align}
    q(1-r)&=-\frac{p(1-r)}{r^2}\left(\frac{\phi_{11}(1-r)}{\phi_{22}(1-r)}\right) \nonumber \\ w(1-r)&=-\frac{p(1-r)}{\phi_{22}(1-r)}. \nonumber
\end{align}
The prefactor in (\ref{eq:p2}) is chosen so that the weight $w$ is equal to $r$ outside the boundary layer. Note that (\ref{eq:SL2}) is not necessary for the numerical comparison in Section \ref{sec:3}, but it helps establish the variational interpretation of the eigenvalues. The following statement provides a few key details regarding these coefficients.

\begin{prop}\label{prop:SL2properties}
    For small enough $\epsilon$, the Sturm-Liouville coefficients have the following properties. The principal coefficient $p$ satisfies:
    \begin{enumerate}[label=\arabic*.]
        \item $p(1-r)$ is continuous and $\frac{1}{p}(1-r)$ is in $L^1(\delta,1)$ for any $\delta>0$.
        \item $p(1-r)$ is positive for $r\in(0,1-r_0)$, negative for $r\in(1-r_0,1]$ and zero when $r=0$ or $r=1-r_0$.
        \item $p(1-r)$ is asymptotic to $\sgn(1-r_0-r)\left(r-1+r_0\right)^{\zeta}$ as $r\to 1-r_0$ with $\zeta=\frac{\phi_{11}(r_0)+(1-r_0)\phi_2(r_0)}{(1-r_0)\phi_{22}'(r_0)}$.
    \end{enumerate}
    The coefficient $q$ satisfies:
    \begin{enumerate}[label=\arabic*.]
        \item $q(1-r)$ is continuous except at $1-r_0$ and in $L^1(\delta,1)$ for any $\delta>0$.
        \item $q(1-r)$ is negative for all $r\in(0,1)$.
        \item $q(1-r)\to-\infty$ as $r\to 1-r_0$.
    \end{enumerate}
    The weight $w$ satisfies:
    \begin{enumerate}[label=\arabic*.]
        \item $w(1-r)$ is continuous except at $1-r_0$ and in $L^1(0,1)$.
        \item $w(1-r)$ is positive for all $r\in(0,1)$.
        \item $w(1-r)\to\infty$ as $r\to 1-r_0$.
    \end{enumerate}
    For all $r\in(0,1)$, $p(1-r)$ converges pointwise to $\frac{r}{8}$, $q(1-r)$ converges pointwise to $-\frac{1}{8r}$, and $w(1-r)$ converges pointwise to $r$ as $\epsilon\to 0$.
\end{prop}

\quad The changing sign of $p$ means that $D_\epsilon$ is mixed-type. Under (\ref{eq:OmegaPM}), (\ref{eq:SL2}), and Proposition \ref{prop:SL2properties}, we identify the elliptic region as the set $M_+=\{(r,\theta): p(1-r)>0\}$ and the singular surface $\Gamma=\{(r,\theta): r=1-r_0\}$ as the interface between elliptic and hyperbolic regions. With a form for $D_\epsilon$ established on the disc, we now prove Theorem \ref{thm:2d} following the same argument as in Section \ref{sec:pT1}.

\subsection{Proof of Theorem \ref{thm:2d}.}\label{sec:pT2} In this section, we use Proposition \ref{prop:phi1phi2V} and the polar representation of $D_\epsilon$ to prove Theorem \ref{thm:2d}. As on the interval, we demonstrate how the Frobenius method determines the regularity of eigenfunctions in the boundary layer and provide a method for finding eigenvalues of $D_\epsilon$ under the $C^2$ regularity condition. 

\quad The coefficients in (\ref{eq:DeDef2}) are radial, meaning the eigenfunctions are necessarily separable. Supposing $u$ satisfies (\ref{eq:PDE}) on the unit disc, we decompose the eigenfunction into radial and angular components $$u(\theta,r)=u^\theta(\theta)u^r(r).$$ The eigenvalue equation (\ref{eq:PDE}) then reduces to a system of ordinary differential equations. The radial component $u^r$ is determined by the equation
\begin{equation}\label{eq:rad}
    \phi_{22}(1-r)(u^r)''(r)+\left(\frac{1}{r}\phi_{11}(1-r)+\phi_2(1-r)\right)(u^r)'(r)-\frac{\nu^2}{r^2}\phi_{11}(1-r) u^r(r)=\lambda u^r(r) 
\end{equation}
while up to a phase shift, the angular component $u^\theta$ is a multiple of $\cos(\nu\theta)$. For $u$ to be continuous, we take $\nu$ to be a nonnegative integer. Because the coefficients in (\ref{eq:DeDef2}) are defined piecewise in the radial coordinate, we decompose the radial eigenfunction $u^r$ into an inner and outer solution. More precisely, if we set
\begin{equation}
    u_I^r(r)=u^r(r)\big|_{[0,1-\epsilon]} \quad \textrm{and} \quad u_O^r(r)=u^r(r)\big|_{[1-\epsilon,1]}, \nonumber
\end{equation}
then (\ref{eq:rad}) can be solved on each subinterval, and a global solution can be determined by gluing local solutions together. Note that (\ref{eq:SL2}) is singular near the origin; this serves as an interior boundary condition, forcing the eigenfunction to vanish at the origin whenever $\nu\neq 0$. The inner solution can be solved explicitly, so we begin our analysis with the more complicated outer solution. In the boundary layer, there are two radial solutions to (\ref{eq:rad}), each with a Frobenius series centered at the singularity $1-r_0$. The following result is analogous to Proposition \ref{lem:0e}, describing the behavior of both local solutions near the edge of the boundary layer.

\begin{prop}\label{lem:0e2}
    Let $(\lambda,u)$ satisfy (\ref{eq:SL2}) on the unit disc. The local radial solution $u^r_O(r)=u^r(r)\big|_{[1-\epsilon,1]}$ is a linear combination of two functions that we denote as $u^r_{\rm{Dir}}$ and $u^r_{\rm{Neu}}$. The solution $u_{\rm{Dir}}^r$ satisfies
    \begin{equation}
        u_{\rm{Dir}}^r(1-r_0)=0 \quad \textrm{and} \quad \lim_{r\to 1-r_0}p(1-r)(u_{\rm{Dir}}^r)'(r)=1 \nonumber
    \end{equation}
    while the solution $u^r_{\rm{Neu}}$ satisfies
    \begin{equation}
        u_{\rm{Neu}}^r(1-r_0)=1 \quad \textrm{and} \quad \lim_{r\to 1-r_0}p(1-r) (u_{\rm{Neu}}^r)'(r)=0. \nonumber
    \end{equation}
    In particular, $u_{\rm{Dir}}^r$ is continuous but not $C^1(1-\epsilon,1)$ while $u_{\rm{Neu}}^r$ is $C^\infty(1-\epsilon,1)$.
\end{prop}

\quad The notation for each basis function stems from their behavior at the singularity $1-r_0$. The condition that $u\in C^2\big(\overline{B_1(0)}\big)$ discards any contribution from the Dirichlet solution, and therefore the outer solution $u_O^r$ is precisely the Neumann solution.

\begin{proof}[Proof of Proposition \ref{lem:0e2}.] Under the transformation $r\to 1-r$, we study (\ref{eq:rad}) on $[0,\epsilon]$. Let $r_0$ be the unique zero of $\phi_{22}$ in this region, introduced in Proposition \ref{prop:phi1phi2V}. Each variable coefficient in (\ref{eq:rad}) is analytic on $(0,\epsilon)$ and has a series representation centered at $r_0$. For the sake of notation, we write
\begin{align}
    \phi_{22}(r)&=\sum_{j=1}^\infty a_j(r-r_0)^j \nonumber \\
    \frac{1}{1-r}\phi_{11}(r)+\phi_2(r)&=\sum_{j=0}^\infty b_j(r-r_0)^j \nonumber \\
    \frac{1}{(1-r)^2}\phi_{11}(r)&=\sum_{j=0}^\infty d_j(r-r_0)^j \nonumber
\end{align}
where $a_1,b_0,$ and $d_0$ are all nonzero. By (\ref{eq:sigma}), each coefficient $a_j, b_j, d_j$ scales with $\epsilon^{-j}$. Any local solution to (\ref{eq:rad}) on $(0,\epsilon)$ takes the form of a Frobenius series
\begin{equation}\label{eq:Frob2}
    u(1-r)=\sum_{j=0}^\infty c_j(r-r_0)^{j+\alpha}
\end{equation}
for $\alpha\in\mathbb{R}$. The indicial polynomial 
\begin{equation}
    P(\alpha)= a_1\alpha(\alpha-1)+b_0\alpha \nonumber
\end{equation}
has two roots at $0$ and $1-\frac{b_0}{a_1}$. By Proposition \ref{prop:phi1phi2V}, $\frac{b_0}{a_1}$ is bounded by a multiple of $\epsilon$ and therefore these roots differ by a non-integer for small enough $\epsilon$. Hence there are two linearly independent solutions to (\ref{eq:rad}) near the boundary that take the form (\ref{eq:Frob2}). Without loss of generality, we initially take $c_0=1$ for both solutions. For $j\geq 1$, the remaining Frobenius coefficients can be determined through the recurrence relation 
\begin{equation}\label{eq:Rec2}
    P(\alpha+j)c_j=\lambda c_{j-1}-\sum_{k=0}^{j-1}c_k \left((k+\alpha)(k+\alpha-1)a_{j-k+1}+(k+\alpha)b_{j-k}-\nu^2d_{j-k-1}\right). 
\end{equation}
\quad When $\alpha=1-\frac{b_0}{a_1}$, we denote the series (\ref{eq:Frob2}) as the Dirichlet solution so that $u^r_{\rm{Dir}}(1-r_0)=0$. Proposition \ref{prop:SL2properties} implies that the quasi-derivative $p(1-r)(u^r_{\rm{Dir}})'(r)$ has a finite, nonzero limit as $r$ approaches $1-r_0$. When $\alpha=0$, we denote the series (\ref{eq:Frob2}) as the Neumann solution so that $p(r_0)(u^r_{\rm{Neu}})'(1-r_0)=0$. By scaling both solutions appropriately, we have the desired result.
\end{proof}

\quad Proposition \ref{lem:0e2} states that the regularity condition $u\in C^2\big(\overline{B_1(0)}\big)$ is equivalent to imposing quasi-Neumann boundary conditions on the elliptic region, just as in Section \ref{sec:1}. Similarly, the spectrum of $D_\epsilon$ then consists entirely of eigenvalues satisfying
\begin{equation}\label{eq:RRQuotient2}
    \lambda=\frac{\displaystyle\int_{0}^{1-r_0} p(1-r)\left|\left(u^r\right)'(r)\right|^2 dr-\nu^2\int_0^{1-r_0}q(1-r)\left|u^r(r)\right|^2 dr}{\displaystyle\int_0^{1-r_0}w(1-r)\left|u^r(r)\right|^2 dr}
\end{equation}
for $u^r$ satisfying (\ref{eq:rad}). By Proposition \ref{prop:phi1phi2V} alongside (\ref{eq:RRQuotient2}), the eigenvalues are real and nonnegative and the eigenfunctions are orthogonal with respect to the $L^2\left(M_+,w\right)$ inner product. Because $w(1-r)$ converges pointwise to $r$ as $\epsilon\to 0$, this corresponds to the fact that the LLE matrix is symmetric and real in the dimensional limit. Further, when $\nu=0$, (\ref{eq:RRQuotient2}) admits a zero eigenvalue, for which the Neumann solution $u_{\rm{Neu}}^r$ is a constant. By (\ref{eq:Pconv}), this verifies that constant vectors are eigenvectors of the LLE matrix and completes the proof of Theorem \ref{thm:2d}.

\quad As in Section \ref{sec:1}, determining eigenvalues via (\ref{eq:RRQuotient2}) would prove difficult. Instead, we rely on the pointwise estimates in Proposition \ref{lem:0e2} and the continuity conditions at the boundary layer. This results in the following statement, which identifies eigenvalues as roots of a transcendental equation.

\begin{prop}\label{prop:DetEq2}
    Let $(\lambda,u)\in[0,\infty)\times C^2\big(\overline{B_1(0)}\big)$ satisfy (\ref{eq:SL2}) with $\nu\in\mathbb{N}_0$. If $u_O^r(r)=u^r(r)\big|_{[1-\epsilon,1]}$ and $\epsilon$ is sufficiently small, then the eigenvalue $\lambda$ satisfies
    \begin{equation}
        J_\nu\left(\sqrt{8\lambda}(1-\epsilon)\right)(u_O^r)'(1-\epsilon)-\sqrt{8\lambda} \: J_\nu'\left(\sqrt{8\lambda}(1-\epsilon)\right)u_O^r(1-\epsilon)=0. \nonumber
    \end{equation}
    Here $J_\nu$ denotes a Bessel function of the first kind.
\end{prop}

\quad This eigenvalue condition appears simpler than that presented in Proposition \ref{prop:DetEq}, when $D_\epsilon$ is defined over the unit interval. This is because the boundary layer in the disc is connected, so there are fewer matching conditions to impose. By Proposition \ref{lem:0e2}, the outer radial solution $u_O^r$ has Neumann boundary conditions on the elliptic region, and its behavior at $1-\epsilon$ can be determined using the Frobenius method.

\begin{proof}[Proof of Proposition \ref{prop:DetEq2}.] Because $u\in C^2\big(\overline{B_1(0)}\big)$, the eigenvalue $\lambda$ must simultaneously verify 
\begin{align}
    \lim_{r\to (1-\epsilon)^-} u_I^r(r)&=\lim_{r\to (1-\epsilon)^+} u_O^r(r)\nonumber \\ \lim_{r\to (1-\epsilon)^-} \left(u_I^r\right)'(r)&=\lim_{r\to (1-\epsilon)^+}\left( u_O^r\right)'(r). \nonumber
\end{align}
By (\ref{eq:rad}) and Proposition \ref{prop:phi1phi2V}, these two conditions guarantee that $u^r$ is $C^2$ in a neighborhood of $1-\epsilon$ and hence $C^2\big(\overline{B_1(0)}\big)$. Because $\lambda$ is nonnegative and (\ref{eq:SL2}) is singular near $r=0$, the inner solution is a bounded Laplacian eigenfunction of frequency $\sqrt{8\lambda}$, i.e. a Bessel function of the first kind. For $u$ to be nontrivial, the following matrix 
\begin{equation}
    \begin{pmatrix} J_\nu\left(\sqrt{8\lambda}(1-\epsilon)\right) & -u_O^r(1-\epsilon) \\ \sqrt{8\lambda}J_\nu'\left(\sqrt{8\lambda}(1-\epsilon)\right) & -\left(u_O^r\right)'(1-\epsilon)\end{pmatrix} \nonumber
\end{equation}
must have zero determinant. This condition equates eigenvalues to roots of the desired transcendental equation.
\end{proof}

\quad In combination, (\ref{eq:Frob2}) and (\ref{eq:Rec2}) alongside Proposition \ref{prop:DetEq2} provide a method for determining eigenvalues of $D_\epsilon$ on the disc. As displayed in Section \ref{sec:3}, the analytic predictions match numerical expectations. The eigenvalues also share a variational interpretation under the requirement that each eigenfunction be in $C^2\big(\overline{B_1(0)}\big)$. If we define the energy functional
\begin{equation}\label{eq:E2}
    E(f)=\frac{\displaystyle \int_0^{1-r_0}p(1-r)\left|f'(r)\right|^2 dr-\nu^2\int_0^{1-r_0} q(1-r)\left|f(r)\right|^2 dr  }{\displaystyle \int_0^{1-r_0} w(1-r)\left|f(r)\right|^2 dr}
\end{equation}
over all $f$ such that $E(f)$ is finite, then the minimizes of $E$ on orthogonal subspaces of $L^2([0,1-r_0],w)$ are solutions of (\ref{eq:rad}) with Neumann boundary conditions \cite[Ch. $6$]{CH}, in the sense that $$\lim_{r\to 1-r_0}p(1-r)f'(r)=0.$$ Note that the integrals in (\ref{eq:E2}) are purely radial because the eigenfunctions of $D_\epsilon$ are separable and any angular contribution scales out. The eigenfunctions are therefore orthogonal in $L^2\left(M_+,w\right)$ and likewise minimize the functional $\langle D_\epsilon f,f\rangle_{L^2(M_+,w)}$ as given in (\ref{eq:Buv}). The Neumann boundary condition matches that established in Section \ref{sec:1} on the unit interval, and hence Theorem \ref{thm:Energy2} holds on the disc.

\section{Numerical Comparison}\label{sec:3}

\quad In this section, we provide numerical evidence of eigenvalue convergence when $D_\epsilon$ is defined on the interval or disc. We also compare eigenvectors of $I-W$ with eigenfunctions of $D_\epsilon$ for small $\epsilon$. The eigenvectors can be explicitly computed with the regularization in (\ref{eq:regularizer}), and the eigenfunctions can be approximated using the Frobenius methods presented in Propositions \ref{lem:0e} and \ref{lem:0e2}.

\quad On the interval, Proposition \ref{prop:DetEq} provides an explicit equation for the analytic eigenvalues in terms of the left solution $u_L$. By imposing quasi-Neumann boundary conditions (\ref{eq:quasiInt}) on the elliptic region, there exists a unique local solution in the boundary layer that can be approximated to any level of precision via (\ref{eq:indPoly}) and (\ref{eq:RecRel}). In this manner, the analytic eigenvalues featured in this section are gathered using a $12$-term Frobenius series representation of $u_L$ in Proposition \ref{prop:DetEq}. Any additional terms in the series result in changes beyond the significant figures presented. 

\quad In Figure \ref{fig:IntLLE}, we present the first $5$ eigenvectors of $I-W$ on the unit interval with $\epsilon=0.05$. The first eigenvector is constant, as anticipated in Section \ref{sec:LLEmatrix}, and the other eigenvectors resemble Laplacian eigenfunctions. In Figure \ref{fig:ErrorInt}, we determine the eigenvalues of $I-W$ on the interval under variance in $n$, the number of sampled data points, and compare with the analytic eigenvalues predicted by Proposition \ref{prop:DetEq}. The relative errors, weighted to mirror Corollary \ref{prop:Est1}, are also featured in Figure \ref{fig:ErrorInt}.

\begin{figure}[H]
    \centering \hspace{11pt}
    \includegraphics[scale=0.98]{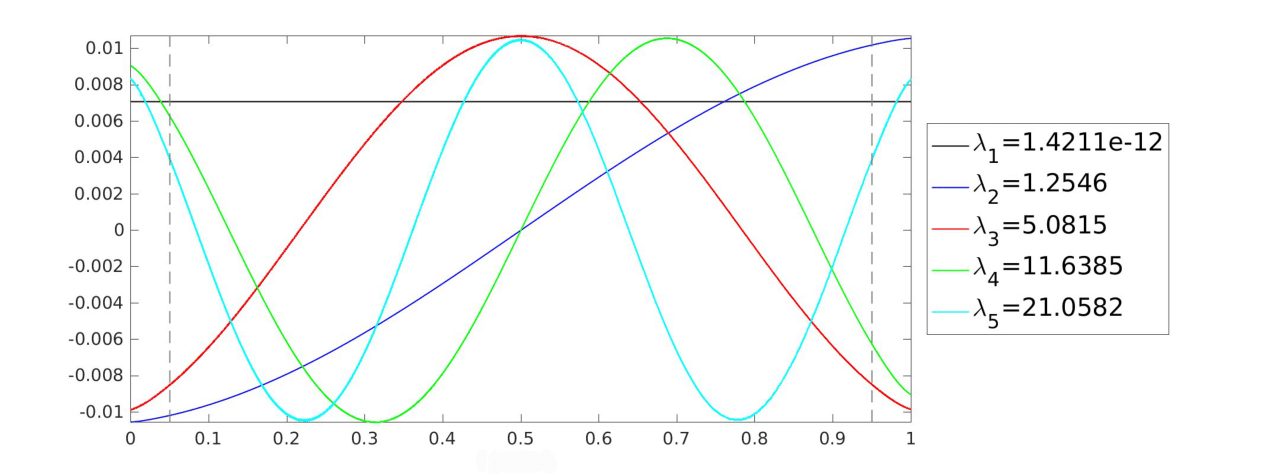}
    \caption{For $\epsilon=0.05$, the first $5$ eigenvectors of $I-W$ sampled on $20,000$ points in the interval are plotted above. Each is labeled with the respective eigenvalue, appropriately scaled by $\epsilon^{-2}$ according to (\ref{eq:Pconv}).}
    \label{fig:IntLLE}
\end{figure}

\begin{rem}
    In \cite[Fig.$5$]{WW}, Wu and Wu display the first $5$ eigenvectors of $I-W$ sampled on $8,000$ points in the interval with $\epsilon=0.01$. They discuss several irregularities that we believe arise due to the small number of data points relative to $\epsilon$. In particular, the eigenvector irregularity in the boundary layer is not present in Figure \ref{fig:IntLLE}.
\end{rem}

\begin{figure}[H]
    \centering
    \includegraphics[scale=0.55]{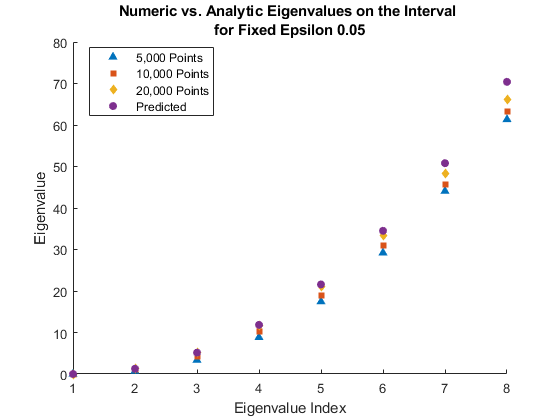}
    \includegraphics[scale=0.55]{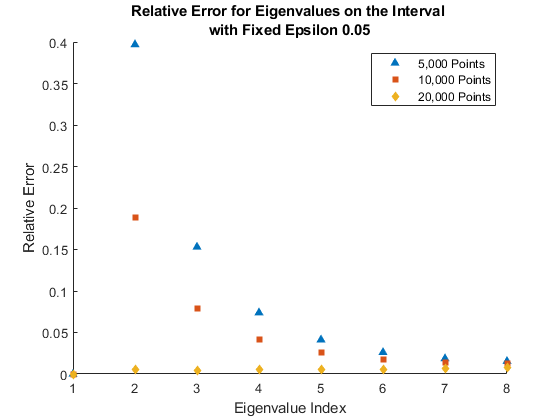}
    \caption{In the left figure, the first $8$ eigenvalues of $I-W$ are plotted for different data set sizes and fixed $\epsilon=0.05$. We consider eigenvalue sequences corresponding to $5,000$ points, $10,000$ points, and $20,000$ points before plotting the analytic expectation to display convergence. In the right figure, we plot the relative errors $\left|\lambda_j(I-W)-\lambda_j(D_\epsilon)\right|\left|\lambda_j(D_\epsilon)\right|^{-3/2}$ for each eigenvalue sequence.}
    \label{fig:ErrorInt}
\end{figure}

\vfill
\pagebreak

\begin{figure}[H]
    \centering
    \includegraphics[scale=0.54]{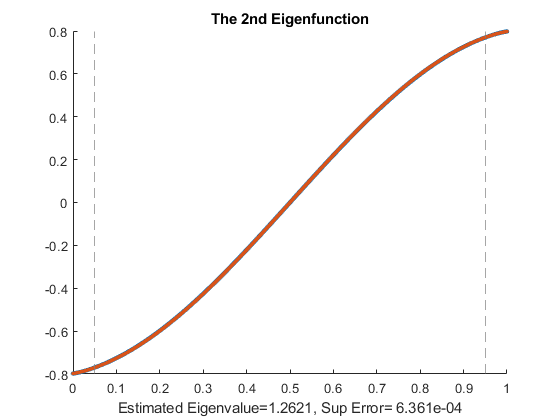} \hfill
    \includegraphics[scale=0.54]{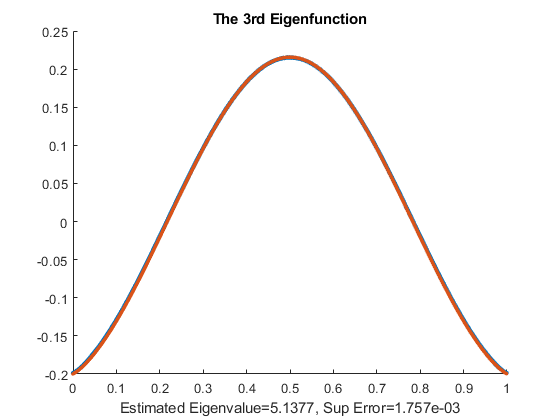} \\ \vspace{10pt}
    \includegraphics[scale=0.54]{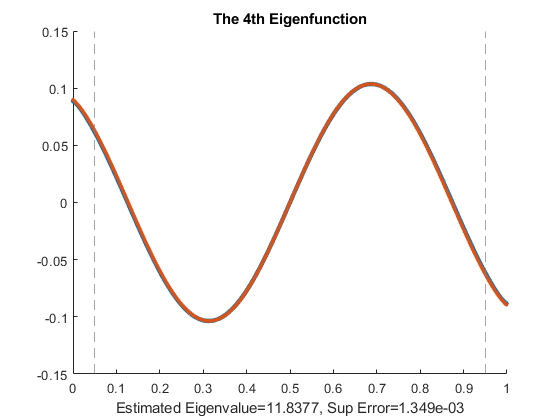} \hfill
    \includegraphics[scale=0.54]{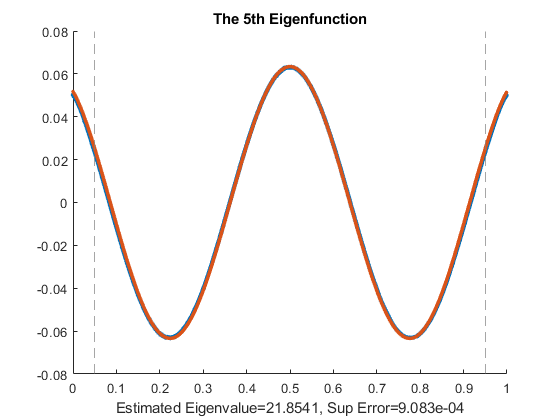}
    \caption{For $\epsilon=0.05$, several eigenfunctions of $D_\epsilon$ with quasi-Neumann boundary conditions (\ref{eq:quasiInt}) plotted above. Each eigenfunction has been discretized over $20,000$ points and is normalized to minimize error with the corresponding eigenvector of $I-W$, featured in Figure \ref{fig:IntLLE}. From top left to bottom right, we display the $2$nd, $3$rd, $4$th, and $5$th eigenfunctions, each labeled with the sup error (when compared with the corresponding eigenvector of $I-W$) and analytically-predicted eigenvalue}
    \label{fig:IntAnalytic}
\end{figure}

\pagebreak

\quad On the disc, Proposition \ref{prop:DetEq2} similarly provides an equation for the analytic eigenvalues in terms of the outer solution $u_O^r$ with quasi-Neumann boundary conditions (\ref{eq:quasiDisc}) on the elliptic region. We use (\ref{eq:Frob2}) and  (\ref{eq:Rec2}) to construct a $12$-term Frobenius series for the local solution in the boundary layer and gather eigenvalues using Proposition \ref{prop:DetEq2}. For comparison, eigenvalues and eigenvectors of $I-W$ are featured in Figure \ref{fig:160k}. Note that each eigenvector appears separable on the disc, with integer angular frequency. The first eigenvector is constant, as anticipated in Section \ref{sec:LLEmatrix}, and the other eigenvectors resemble Laplacian eigenfunctions on the disc.

\begin{figure}[H]
    \centering
    \includegraphics[scale=0.57]{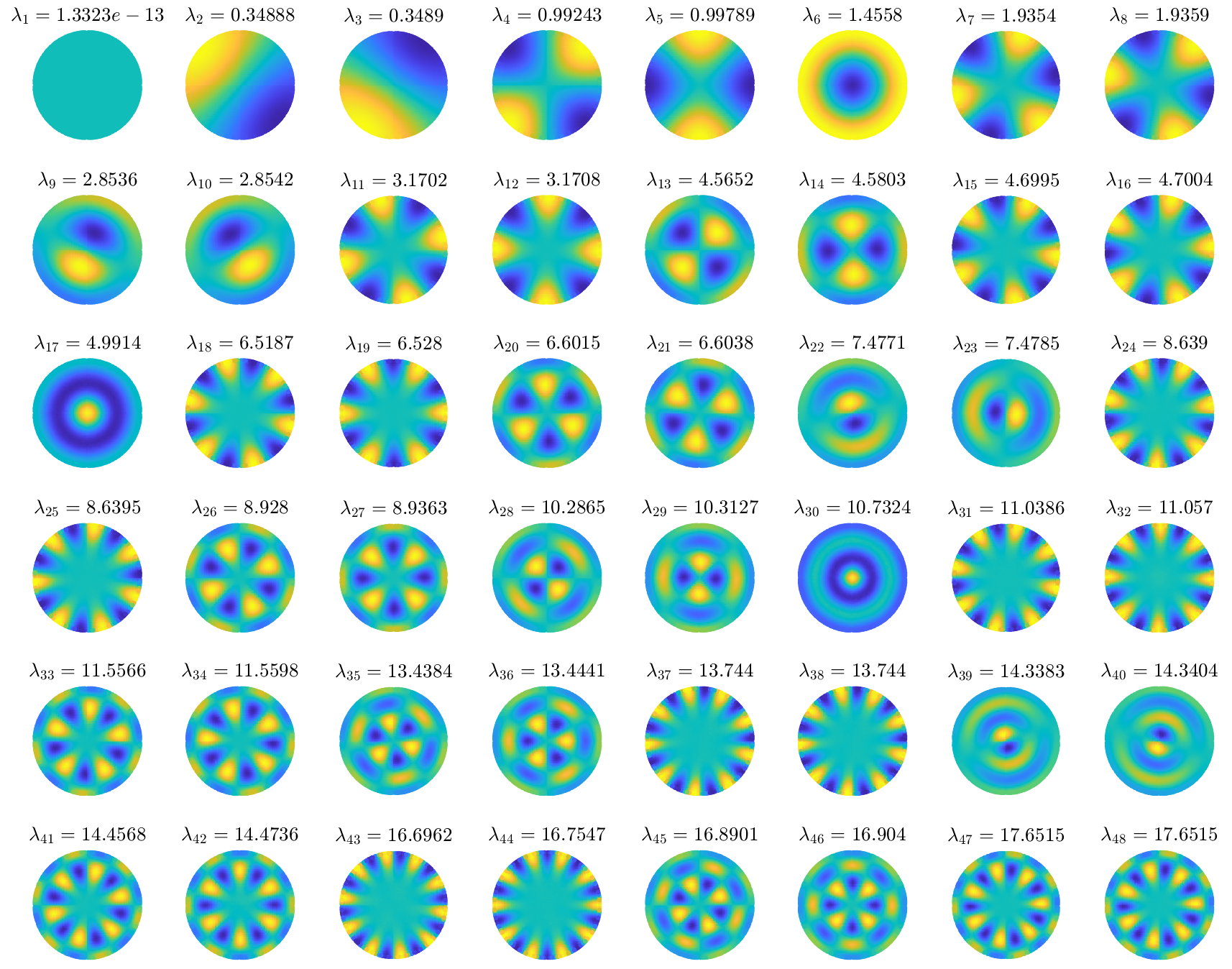}
    \caption{For $\epsilon=0.05$, the first $48$ eigenvectors of $I-W$ sampled on $160,000$ points in the disc are plotted from top left to bottom right. Each is labeled with the respective eigenvalue, appropriately scaled by $\epsilon^{-2}$ according to (\ref{eq:Pconv}). }
    \label{fig:160k}
\end{figure}

\quad In Figure \ref{fig:0.05}, we let $\epsilon=0.05$ and test an increasing sequence in $n$, the number of data points from which the LLE matrix $W$ is constructed. In particular, we let $\{x_j\}_{j=1}^n$ form a grid in $[-1,1]\times[-1,1]$ and remove points for which the norm is greater than $1$. Note that by taking $n$ too large, we leave the supercritical regime (\ref{eq:supercrit}), and the eigenvectors lose any resemblance to separable functions on the disc. To accompany the evidence of eigenvalue convergence, we also display several eigenfunctions of $D_\epsilon$ in Figure \ref{fig:FtoV}.

\vfill
\pagebreak

\begin{figure}[H]
    \centering
    \includegraphics[scale=0.54]{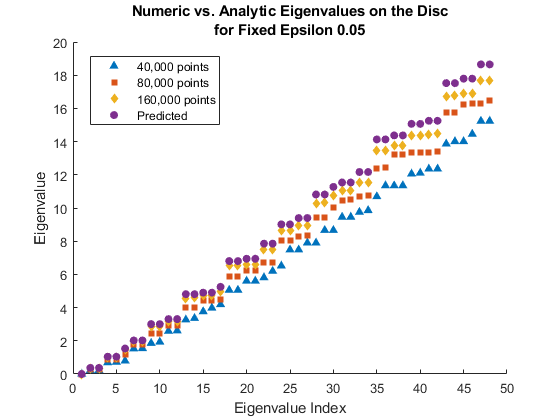} \hspace{4pt}
    \includegraphics[scale=0.54]{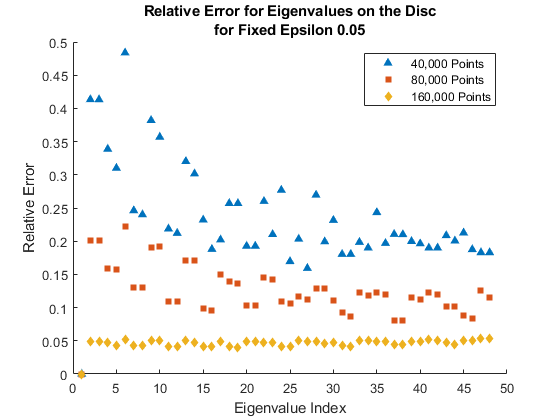}
    \caption{In the left figure, the first $48$ eigenvalues of $I-W$ are plotted for different data set sizes and fixed $\epsilon=0.05$. We consider eigenvalue sequences corresponding to $40,000$ points, $80,000$ points, and $160,000$ points before plotting the analytic expectation to display convergence. In the right figure, we plot the standard relative errors  $\left|\lambda_j(I-W)-\lambda_j(D_\epsilon)\right|\left|\lambda_j(D_\epsilon)\right|^{-1}$ for each eigenvalue sequence.}
    \label{fig:0.05}
\end{figure}
\vspace{-10pt}

\begin{figure}[H]
    \centering
    \includegraphics[scale=0.55]{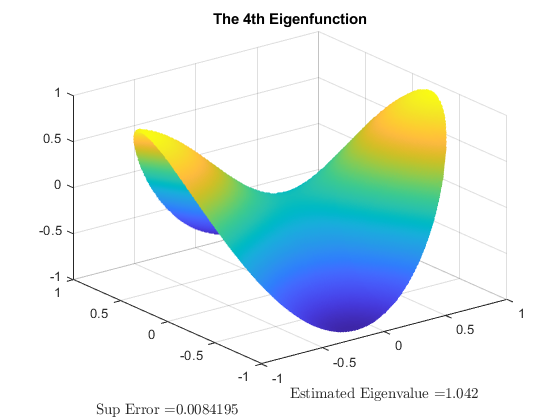}
    \hfill
    \includegraphics[scale=0.55]{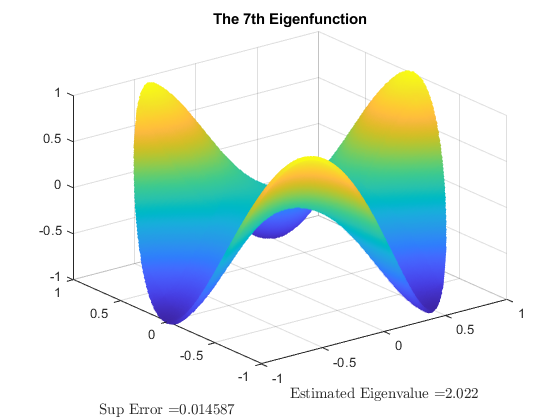} 
    \\ \vspace{10pt}
    \includegraphics[scale=0.55]{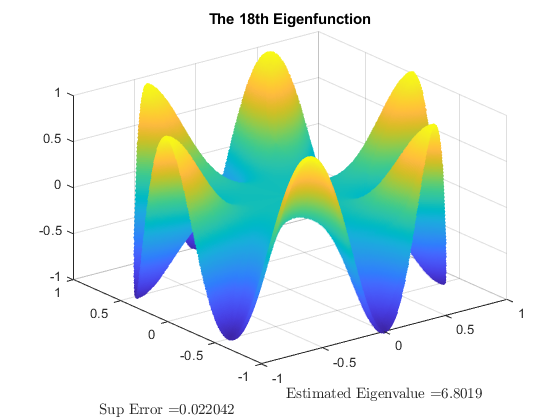}
    \hfill
    \includegraphics[scale=0.55]{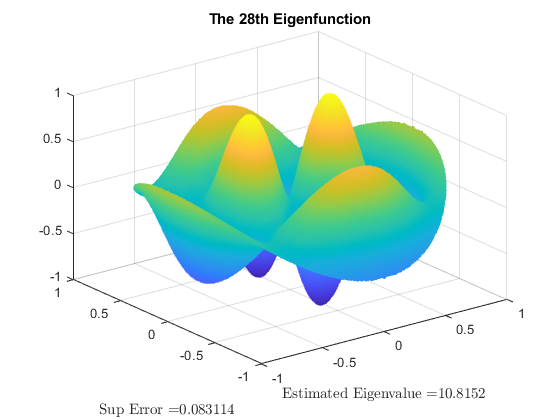}
    \caption{For $\epsilon=0.05$, several eigenfunctions of $D_\epsilon$ with quasi-Neumann boundary conditions (\ref{eq:quasiDisc}) are plotted above. Each eigenfunction has been discretized over $160,000$ points and is rotated and normalized to minimize error with the corresponding eigenvector of $I-W$, featured in Figure \ref{fig:160k}. 
    From top left to bottom right, we display the $4$th, $7$th, $18$th, and $28$th eigenfunctions, each labeled with the sup error (when compared with the corresponding eigenvector of $I-W$) and analytically-predicted eigenvalue.}
    \label{fig:FtoV}
\end{figure}
\vspace{-10pt}

\section{Proof of Theorem \ref{thm:blah}.}\label{sec:6}

\quad In this section, we use a technique from calculus of variations \cite[Ch. $7$]{BVB} to recover quasi-Neumann boundary conditions from a min-max principle governed by Definition \ref{ass:form} and prove Theorem \ref{thm:blah}. Let $M$ be a smooth, compact Riemannian manifold with smooth boundary and let $D_\epsilon$ be admissible on $M$. Unless otherwise stated, we work in the Hilbert space $L^2(M_+,w)$ where $M_+$ is the elliptic region in (\ref{eq:OmegaPM}) and $w$ is given in Definition \ref{ass:form}. Set
\begin{equation}
    V=\{v\in L^2(M_+,w): \langle D_\epsilon v,v\rangle<\infty\} \nonumber
\end{equation}
and suppose that the $j$th eigenvalue of $D_\epsilon$ on $M$ can be written
\begin{equation}
    \lambda_j=\max_{v_1,\dots, v_{j-1}\in L^2(M_+,w)}\left\{\min_{v\in V\cap \{v_1,\dots, v_{j-1}\}^\perp\backslash\{0\}}\frac{\langle D_\epsilon v,v\rangle}{\langle v,v\rangle}\right\}. \nonumber
\end{equation}
Given \ref{eq:Buv}, the minimum over each orthogonal subspace is attained by solutions to (\ref{eq:PDE}) on $M_+$. Let $u$ be an eigenfunction of $D_\epsilon$ with $||u||=1$ and set
\begin{align}
    J_v(t)=\frac{\langle D_\epsilon(u+tv),u+tv\rangle}{||u+tv||} \nonumber
\end{align}
for $v\in V\cap \{v_1,\dots, v_{j-1}\}^\perp\backslash\{0\}$. Then, because $u$ is a minimizing function and $A_\epsilon$ is symmetric,
\begin{align}\label{eq:Jt=0}
    \frac{1}{2}\frac{d}{dt}J_v(t)\big|_{t=0}= \Re \langle D_\epsilon u,v\rangle - J_v(0)\Re \langle u,v\rangle=0.
\end{align}

\quad In particular, (\ref{eq:Jt=0}) holds true for arbitrary $v\in V\cap \{v_1,\dots, v_{j-1}\}^\perp\backslash\{0\}$. By separating (\ref{eq:Jt=0}) into interior and boundary integrals, we recover an Euler Lagrange equation dictating (\ref{eq:PDE}) on $M_+$ with $\lambda=J_v(0)$. To determine the boundary conditions that the eigenfunctions satisfy, consider the first term in (\ref{eq:Buv}),
\begin{equation}
\begin{aligned}\label{eq:first}
    \frac{1}{2(d+2)}\int_{M\backslash M_\epsilon} \nabla u(x)\overline{\nabla v(x)}dv_g =&-\frac{1}{2(d+2)}\int_{M\backslash M_\epsilon}\Delta u(x) \overline{v(x)}dv_g \\ &+\frac{1}{2(d+2)}\int_{\p M_\epsilon} \frac{\p u}{\p \nu}(x) \overline{v(x)}d\sigma
\end{aligned}
\end{equation}
where $d\sigma$ is the measure on $\p M_\epsilon$ and $\nu$ is the outward-pointing unit normal along the boundary. According to \cite{WW}, $D_\epsilon$ is equal to the Laplacian scaled by $\frac{1}{2(d+2)}$ outside the boundary layer, and so the interior integral in (\ref{eq:first}) agrees with $\langle D_\epsilon u,v\rangle_{L^2(M\backslash M_\epsilon)}$. Meanwhile, by applying integration by parts to the second term in (\ref{eq:Buv}), we find
\begin{equation}\label{eq:second}
\begin{aligned}
    \int_{M_+\cap M_\epsilon}p(x_d)\frac{\p u}{\p x_d}(x)\overline{\frac{\p v}{\p x_d}}(x)dx=&-\int_{M_+\cap M_\epsilon} \frac{\p}{\p x_d}\left(p(x_d)\frac{\p u}{\p x_d}(x)\right)\overline{v(x)}dx  \\
    &+p(\epsilon)\int_{\p M_\epsilon}  \frac{\p u}{\p x_d}(x',\epsilon) \overline{v(x',\epsilon)}dx' \\ &- p(x_+)\int_{\p M_+}\frac{\p u}{\p x_d}(x',x_+)\overline{v(x',x_+)}dx'
\end{aligned}
\end{equation}
where $(x',x_d)$ denotes the Fermi coordinates in $M_\epsilon$ and $x_+$ denotes the $d$th Fermi coordinate for points along $\p M_+$. The interior integral in (\ref{eq:second}) captures the degenerate behavior of $D_\epsilon$ in $M_\epsilon$. Because $D_\epsilon$ is admissible, the boundary integrals in (\ref{eq:first}) and (\ref{eq:second}) sum to
\begin{equation}\label{eq:bdryint}
    -p(x_+)\int_{\p M_+}\frac{\p u}{\p x_d}(x',x_+) \overline{v(x',x_+)}dx'. 
\end{equation}
By construction, $-\frac{\p}{\p x_d}$ is the (outward) normal derivative on $M_+$, and by (\ref{eq:Jt=0}), we can set (\ref{eq:bdryint}) equal to zero. Because $v$ is arbitrary, we conclude that $$p(x)\frac{\p u}{\p x_d}(x)=0$$ for $x$ along $\p M_+$. Thus, $u$ has quasi-Neumann boundary conditions on the elliptic region (\ref{eq:OmegaPM}), and Theorem \ref{thm:blah} holds.

\appendix\section{A Variational Convergence Example}\label{sec:appendix}

\quad In this section, we demonstrate how the higher order correction in (\ref{eq:Pconv}) might force a regularity condition on the eigenfunctions of the leading order operator $D_\epsilon$ on the unit interval. As shown in Lemma \ref{lem:0e}, absent boundary conditions, there are two linearly independent solutions to (\ref{eq:PDE}) in the boundary layer. Both solutions are smooth almost everywhere, meaning there is zero probability of sampling either function at a point of low regularity. However, by taking $\epsilon$ to be supercritical with respect to the sample size (\ref{eq:supercrit}), it is guaranteed that a subset of the boundary layer will be sampled in the LLE algorithm. We provide a condition for which eigenvalue convergence may be penalized by the presence of singular points near this subset. 

\quad According to (\ref{eq:Pconv}), the matrix $I-W$ behaves like the differential operator $D_\epsilon$ with a correction. To the author's knowledge, a precise description of this correction is unknown, although the derivation in \cite[Appendix F.]{WW} indicates that it is a linear differential operator of degree no less than three. Without any imposed regularity, consider the differential operator $\widetilde{D_\epsilon}$ given by
\begin{equation}\label{eq:Detilde}
    \widetilde{D_\epsilon}f(x)\eqdef\phi_{2}(x)f''(x)+\phi_1(x)f'(x)+\epsilon\sum_{j=3}^J\phi_j(x) f^{(j)}(x)
\end{equation}
where $\phi_1,\phi_2$ are the piecewise functions in (\ref{eq:DeDef}) and $J\in\mathbb{N}$. Suppose each $\phi_j$ is uniformly bounded on $[0,1]$ and satisfies the symmetry condition $$\phi_j(x)=(-1)^j\phi_j(1-x).$$

\quad This operator is constructed to be a perturbation of $D_\epsilon$ in (\ref{eq:DeDef}) and remain consistent with properties of the LLE matrix discussed in Section \ref{sec:LLEmatrix}. If $u$ is an eigenfunction of $\widetilde{D_\epsilon}$ in the weak sense, normalized in $L^2(M_+,w)$, then the corresponding eigenvalue $\tilde{\lambda}$ can be written
\begin{equation}\label{eq:almostEnergy}
     -\int_{x_0}^{1-x_0}\left(p(x)u'(x)\right)'\overline{u(x)}\:dx +\epsilon \sum_{j= 3}^J\displaystyle \int_{x_0}^{1-x_0}w(x) \phi_j(x)u^{(j)}(x)\overline{u(x)}\:dx
\end{equation}
where $p,w$ are given in Proposition \ref{prop:SLproperties}. Note that, so long as each integral in (\ref{eq:almostEnergy}) is finite, $\tilde{\lambda}$ begins to resemble (\ref{eq:RRQuotient}) as $\epsilon\to 0$. However, there are conditions on the coefficients $\phi_j$ in (\ref{eq:Detilde}) for which only one function in Proposition \ref{lem:0e} allows (\ref{eq:almostEnergy}) to remain finite.

\begin{assumption}\label{ass:PsiJ}
    There exists some $j\in\{3,\dots, J\}$ and some $0\leq k\leq j-2+(4+2\sqrt{3})\epsilon$ such that $$\lim_{x\to x_0^+} \phi_j(x)(x-x_0)^{-k}$$ is finite and nonzero.
\end{assumption}

\quad Under this assumption, integrability issues arise in (\ref{eq:almostEnergy}) whenever $u$ behaves asymptotically like $u_{\rm{Dir}}$ from Proposition \ref{lem:0e} near $x_0$. This is made precise by the following statement. 

\begin{lemma}\label{lem:condition}
    Let $f$ be defined as $$f(x)=\sin\beta \: u_{\rm{Dir}}(x)+\cos\beta \: u_{\rm{Neu}}(x)$$ for some $\beta\in[0,2\pi)$. Then $$\langle \widetilde{D_\epsilon}f,f\rangle_{L^2(M_+,w_1)}<\infty$$ if and only if $\beta=0$.
\end{lemma}

\quad Under Assumption \ref{ass:PsiJ}, the eigenvalues of $\widetilde{D_\epsilon}$ converge variationally to the eigenvalues of $D_\epsilon$ if and only if Neumann boundary conditions are imposed on the elliptic region in the limit. Equivalently, this requires eigenfunctions in $C^2([0,1])$.

\begin{proof}[Proof of Lemma \ref{lem:condition}.] Let $f$ be a linear combination of the Dirichlet and Neumann solutions parametrized by $\beta\in[0,2\pi)$. By Propositions \ref{lem:0e} and \ref{prop:SLproperties},
\begin{align}
    \int_{x_0}^{1-x_0}\left(p(x)f'(x)\right)'\overline{f(x)}dx <\infty \nonumber
\end{align}
for all $\beta$. However, while
\begin{equation}
    \sum_{j=3}^J\int_{x_0}^{1-x_0} w(x) \phi_j(x) u_{\rm{Neu}}^{(j)}(x)\overline{u_{\rm{Neu}}(x)}\:dx<\infty, \nonumber
\end{equation}
the same integral blows up if it features any contribution from the Dirichlet solution $u_{\rm{Dir}}$, for which $u_{\rm{Dir}}^{(j)}(x)$ is asymptotic to $(x-x_0)^{1-(4+2\sqrt{3})\epsilon-j}$ as $x\to x_0^+$. 
\end{proof}

\quad On the interval, we suspect that the higher order terms in (\ref{eq:Pconv}) takes the form given in (\ref{eq:Detilde}) and that Assumption \ref{ass:PsiJ} holds. For other manifolds, we state the following conjecture regarding the higher order operator.

\begin{conj}
    The $C^2$ regularity condition, as implemented in Theorems \ref{thm:1d} and \ref{thm:2d}, arises from the weak form of the full operator in (\ref{eq:Pconv}).
\end{conj}


\bibliographystyle{plain}
\bibliography{LLEbib}


\end{document}